\documentclass[12pt]{article}
\usepackage{float}       
\usepackage{placeins}    
\usepackage{caption}     
\usepackage{tikz}
\usepackage{enumerate}
\usepackage[all,cmtip]{xy}
\usetikzlibrary{arrows,matrix,positioning}
\usetikzlibrary{arrows.meta}
\usepackage{amsmath}
\usepackage{xcolor}
\usepackage{tikz}
\usetikzlibrary{arrows.meta}

\definecolor{routegray}{gray}{0.82}

\usepackage{amsmath}
\usepackage{amssymb}
\usepackage[table]{xcolor}
\usepackage{ytableau}

\definecolor{routegray}{gray}{0.82}

\ytableausetup{
  boxsize=1.8em,
  centertableaux
}
\newcommand{\deletedbox}{*(gray!55)\phantom{0}}
\newcommand{\jdtbox}{*(black)\phantom{0}}

\usepackage{soul}
\tikzset{>=stealth',
  head/.style = {fill = white, text=black},
  plaque/.style = {draw, rectangle, minimum size = 10mm},
  pil/.style={->,thick},
  junct/.style = {draw,circle,inner sep=0.5pt,outer sep=0pt, fill=black}
  }

\usepackage[english]{babel}
\usepackage{ytableau}
\usepackage[a4paper,top=2cm,bottom=2cm,left=3cm,right=3cm,marginparwidth=1.75cm]{geometry}

\usepackage{amssymb}
\usepackage{amsmath}
\usepackage[title]{appendix}

\usepackage{siunitx}
\PassOptionsToPackage{hyphens}{url}\usepackage[colorlinks=true, pdfstartview=FitV, linkcolor=blue, citecolor=blue, urlcolor=blue]{hyperref}
\usepackage{cleveref}
\usepackage[utf8]{inputenc}
\usepackage{csquotes}
\usepackage{booktabs}
\usepackage{longtable}
\usepackage{adjustbox}
\usepackage{array}
\usepackage{url}
\usepackage{titlesec}
\usepackage{authblk}
\usepackage{xcolor} 
\usepackage{mathtools}
\numberwithin{equation}{section}

\newcommand{\ctab}[1]{\vcenter{\hbox{#1}}}

\usepackage{amsthm}

\newtheoremstyle{normalnote}
  {3pt}   
  {3pt}   
  {\itshape} 
  {}      
  {\bfseries} 
  {.}     
  {.5em}  
  {\thmname{#1}\thmnumber{ #2}\thmnote{ {\normalfont(#3)}}}

\theoremstyle{normalnote}
\newtheorem{theorem}{Theorem}[section]
\newtheorem{lemma}[theorem]{Lemma}
\newtheorem{proposition}[theorem]{Proposition}
\newtheorem{corollary}[theorem]{Corollary}

\newenvironment{myproof}
{\noindent\textit{Proof.}}
{\hfill$\blacksquare$}
\titleformat{\subsection}
  {\mdseries\itshape\large} 
  {\thesubsection}{1em}{} 

\crefformat{figure}{#2Figure~#1#3}
\Crefformat{figure}{#2Figure~#1#3}
\crefformat{table}{#2Table~#1#3}
\Crefformat{table}{#2Table~#1#3}
\crefformat{section}{#2Section~#1#3}
\Crefformat{section}{#2Section~#1#3}

\usepackage{soul}
\usepackage{xcolor}
\newtheorem{definition}[theorem]{Definition}
\newtheorem{remark}[theorem]{Remark}%

\newtheorem{conjecture}[theorem]{Conjecture}

\DeclareMathOperator{\co}{co}

\DeclareMathOperator{\SSYT}{SSYT}
\DeclareMathOperator{\SVT}{SVT}
\DeclareMathOperator{\RPP}{RPP}
\DeclareMathOperator{\SVRPP}{SVRPP}
\DeclareMathOperator{\ceq}{ceq}
\DeclareMathOperator{\wt}{wt}
\DeclareMathOperator{\ex}{ex}

\DeclareMathOperator{\OS}{OS}

\def\l{{\lambda}}
\def\m{{\mu}}

\def\qed{{\hfill$\blacksquare$}}

\newcommand{\ircont}{\mathrm{ircont}}

\begin{document}

\begin{center}
{\Large\bf A  Lam--Postnikov--Pylyavskyy   inequality for hybrid Grothendieck polynomials}

\vskip 6mm
{\small   }
Peter L. Guo, Mingyang Kang, Jiaji Liu

\end{center}

\begin{abstract}
We prove  a multivariate Lam--Postnikov--Pylyavskyy type inequality for hybrid Grothendieck polynomials, unifying and refining   results for stable and dual stable Grothendieck polynomials    established   by  Chan--Chen--Pak--Soskin. 
We also conjecture  extensions   of  the    Lam--Postnikov--Pylyavskyy inequality  and   a  conjecture by Thomas--Yong to the (equivariant) Schubert and Grothendieck polynomial setting.


\end{abstract}

\section{Introduction}\label{sec:introduction}
\subsection{Lam--Postnikov--Pylyavskyy  inequality}

The Lam--Postnikov--Pylyavskyy (LPP) inequality \cite{LamPostnikovPylyavskyy2007} is a celebrated result  in the Schur positivity and log-concavity phenomenon. For partitions $\lambda=(\lambda_1\geq \lambda_2\geq \cdots)$ and $\mu=(\mu_1\geq \mu_2\geq \cdots)$,  define their join $\lambda\vee\mu$ and meet $\lambda\wedge\mu$ as 
\[
  (\lambda\vee\mu)_i=\max\{\lambda_i,\mu_i\},
  \qquad
  (\lambda\wedge\mu)_i=\min\{\lambda_i,\mu_i\}.
\]
Let $s_\lambda(\mathbf x)$ be the Schur function associated with $\lambda$. For two symmetric functions $f$ and $g$, write $f\leq_s g$ if $g-f$ is Schur positive, namely,
\[
g-f\in \sum_\lambda \mathbb R_{\ge0} \cdot s_\lambda(\mathbf x).
\]
The Schur LPP inequality states that 

\begin{theorem}[\text{\cite[Theorem 5]{LamPostnikovPylyavskyy2007}}]\label{llp-kk}
For   partitions $\lambda$ and $\mu$, we have 
\[
   s_\lambda(\mathbf x)s_\mu(\mathbf x)
   \leq_s
   s_{\lambda\vee\mu}(\mathbf x)s_{\lambda\wedge\mu}(\mathbf x).
\]
\end{theorem}

For more thorough information about LPP type inequalities as well as general Schur positivity and log-concavity, we refer  to  the recent breakthrough  work by Speyer \cite{Speyer2026} and  Chan--Chen--Pak--Soskin \cite{CCPS}.

As an advanced  extension of the LPP inequality, Chan, Chen, Pak and Soskin \cite[Theorem 1.6]{CCPS} developed the Schur positivity version of the Ahlswede--Daykin inequality. As the main application, they resolved a Schur positivity conjecture  concerning stable Grothendieck polynomials   made  by  Mihalcea. Stable Grothendieck polynomials $G_\lambda(\mathbf{x})$   are the central  combinatorial  object  in the study of Schubert calculus in the K-theory of Grassmannians. A notable combinatorial construction of  $G_\lambda(\mathbf{x})$ was given   by  Buch \cite{Buch2002}   by using   set-valued tableaux.

Let  $\tilde{G}_\lambda(\mathbf{x})=(-1)^{|\lambda|}G_\lambda(-\mathbf{x})$  be the modification such that each monomial has nonnegative coefficient. The following was   conjectured by Mihalcea \cite{Mihalcea}.

\begin{theorem}[\text{\cite[Theorem 1.8]{CCPS}}]\label{thm-ek}
For partitions $\lambda$ and $\mu$, we have 
\[
   \widetilde{G}_\lambda(\mathbf x)\widetilde{G}_\mu(\mathbf x)
   \leq_s
   \widetilde{G}_{\lambda\vee\mu}(\mathbf x)\widetilde{G}_{\lambda\wedge\mu}(\mathbf x).
\]
\end{theorem}

Recall that the Hall inner product on symmetric functions is defined by letting $<s_\lambda, s_\mu>=\delta_{\lambda, \mu}$. The orthogonal  dual to  $G_\lambda(\mathbf{x})$, denoted $g_\lambda(\mathbf{x})$, was systematically  investigated  by   Lam and Pylyavskyy \cite{LamPylyavskyy2007} which can be combinatorially generated by  reverse plane partitions. The following dual stable  Grothendieck LPP inequality  is another application of the Ahlswede--Daykin--Schur inequality.

\begin{theorem}[\text{\cite[Theorem 1.10]{CCPS}}]\label{thm-ek-d}
For  partitions $\lambda$ and $\mu$, we have 
\[
   g_\lambda(\mathbf x)g_\mu(\mathbf x)
   \leq_s
   g_{\lambda\vee\mu}(\mathbf x)g_{\lambda\wedge\mu}(\mathbf x).
\]
\end{theorem}

\subsection{Main result}

The hybrid Grothendieck polynomial  $H_\lambda(\mathbf{x}; \mathbf{t}, \mathbf{w})$ with two families of  parameters $\mathbf{t}$ and $\mathbf{w}$ is the generating function 
\[
 H_{\lambda}(\mathbf{x};\mathbf{t},\mathbf{w})
  =
  \sum_{T\in\SVRPP(\lambda)}
\mathbf{x}^{\ircont(T)}\mathbf{t}^{\ceq(T)}\mathbf{w}^{\ex(T)}
\]
over set-valued reverse plane partitions of shape $\lambda$. This class of symmetric functions was introduced by the authors \cite{GuoKangLiu2025} originally in order to unify stable and dual stable Grothendieck polynomials. More specifically,   

\begin{remark}\label{rmk-agh}
We have  the following specializations.
 \begin{itemize}

\item   $t_i=0$ for $i\geq 1$:    ${H}_{\lambda}(\mathbf{x};\mathbf{0},\mathbf{w})$ is   the refined stable Grothendieck polynomial $G_\lambda(\mathbf x; \mathbf w)$  defined  by Chan  and Pflueger \cite{ChanPflueger2021}. If  further setting  every  $w_i=-1$, then ${H}_{\lambda}(\mathbf{x};\mathbf{0},\mathbf{-1})$ is  equal to $G_\lambda(\mathbf{x})$. Note also that  ${H}_{\lambda}(\mathbf{x};\mathbf{0},\mathbf{1})=\widetilde{G}_\lambda(\mathbf x)$.

\item   $w_i=0$ for $i\geq 1$:     ${H}_{\lambda}(\mathbf{x};\mathbf{t},\mathbf{0})$ is  the refined dual stable Grothendieck polynomial $g_\lambda(\mathbf x;\mathbf t)$  defined   by Galashin, Grinberg and Liu \cite{GalashinGrinbergLiu2016}. If further setting  every  $t_i=1$, then ${H}_{\lambda}(\mathbf{x};\mathbf{1},\mathbf{0})$ is   equal to $g_\lambda(\mathbf{x})$.

\end{itemize}   
\end{remark}
Though known properties suggest that  $H_\lambda(\mathbf{x}; \mathbf{t}, \mathbf{w})$ possesses remarkable properties, resembling those owned by stable and dual stable Grothendieck polynomial, some problems and conjectures are still open  in the hybrid setting \cite{GuoKangLiu2025}.

We    extend the ordinary  Schur positivity to a polynomial level. Suppose that $f$ and $g$ are symmetric functions in $\mathbf{x}$, with coefficients being polynomials in $\mathbf{z}=(z_1,z_2,\ldots)$. 
Write $f\leq_{s}^{\mathbf z}g$
if
\[
g-f\in \sum_\lambda \mathbb R_{\ge0}[\mathbf z] \cdot s_\lambda(\mathbf x).
\]

Our main result is stated as follows.  

\begin{theorem}\label{main-1}
 For partitions $\lambda$ and $\mu$, we have   \begin{equation}\label{eq:intro-Hp-lpp}
H_\lambda(\mathbf{x};\mathbf{t}, \mathbf{w})
   H_\mu(\mathbf{x};\mathbf{t},\mathbf{w})
   \leq_{s}^{\mathbf t,\mathbf w}
   H_{\lambda\vee\mu}(\mathbf{x};\mathbf{t}, \mathbf{w})
   H_{\lambda\wedge\mu}(\mathbf{x};\mathbf{t},\mathbf{w}).
\end{equation} 
\end{theorem}

By Remark \ref{rmk-agh}, we see that  Theorem \ref{main-1} unifies Theorems \ref{thm-ek} and \ref{thm-ek-d} for stable and dual stable Grothendieck polynomials.

Taking $\mathbf t=\mathbf 0$ in $H_\lambda(\mathbf{x}; \mathbf{t}, \mathbf{w})$  gives the refined stable Grothendieck polynomial $G_\lambda(\mathbf x; \mathbf w)$ of Chan  and Pflueger \cite{ChanPflueger2021}. In this case,  Theorem \ref{main-1} leads to a refinement  of Theorem \ref{thm-ek}.

\begin{corollary}[refined version of Theorem \ref{thm-ek}]\label{k-0-p}
 For partitions $\lambda$ and $\mu$, we have   \begin{equation*} 
G_\lambda(\mathbf{x};  \mathbf{w})
   G_\mu(\mathbf{x}; \mathbf{w})
   \leq_{s}^{\mathbf w}
  G_{\lambda\vee\mu}(\mathbf{x}; \mathbf{w})
   G_{\lambda\wedge\mu}(\mathbf{x}; \mathbf{w}).
\end{equation*} 
\end{corollary}

Accordingly, taking   $\mathbf w=\mathbf 0$ in Theorem \ref{main-1}  yields  a LPP inequality for  the refined dual stable Grothendieck polynomial $g_\lambda(\mathbf x;\mathbf t)$ of Galashin, Grinberg and Liu \cite{GalashinGrinbergLiu2016}. 

\begin{corollary}[refined version of Theorem \ref{thm-ek-d}]\label{k-1-p}
 For  partitions $\lambda$ and $\mu$, we have   \begin{equation*} 
g_\lambda(\mathbf{x};\mathbf{t})
   g_\mu(\mathbf{x};\mathbf{t})
   \leq_{s}^{\mathbf t}
   g_{\lambda\vee\mu}(\mathbf{x};\mathbf{t})
   g_{\lambda\wedge\mu}(\mathbf{x};\mathbf{t}).
\end{equation*} 
\end{corollary}

Note that both Corollaries  \ref{k-0-p} and \ref{k-1-p} specialize to Theorem \ref{llp-kk} by taking $\mathbf w=\mathbf 0$ and $\mathbf t=\mathbf 0$, respectively.

Our proof of  Theorem \ref{main-1}  seems highly nontrivial.  The arguments for Theorems \ref{thm-ek} and \ref{thm-ek-d} in \cite{CCPS} do not seem to directly apply to Theorem \ref{main-1}. Specifically, we will encounter  two difficulties which will be briefly explained as follows.   Write 
\[
H_\lambda(\mathbf{x};\mathbf{t},\mathbf{w})
   =
   \sum_\mu K_{\lambda,\mu}(\mathbf{t},\mathbf{w})s_\mu(\mathbf{x}).
\]
\begin{itemize}
    \item  The proofs of Theorems \ref{thm-ek} and \ref{thm-ek-d} technically rely  on    tableau formulas for the coefficients in the Schur expansion of stable or dual stable Grothendieck polynomials \cite{CCPS}. There exists a combinatorial formula for the coefficients $K_{\lambda,\mu}(\mathbf{t},\mathbf{w})$   using the language of lattice words   \cite[Theorem 1.3]{GuoKangLiu2025}.  Unfortunately, this  formula for $K_{\lambda,\mu}(\mathbf{t},\mathbf{w})$ does not allow us to use similar arguments  as in the  proofs of Theorems \ref{thm-ek} and \ref{thm-ek-d}. The essential obstruction is that the lattice word model  lacks the structure of a distributive lattice. 

    \item Another key tool used to prove  Theorems \ref{thm-ek} and \ref{thm-ek-d} is  the  Ahlswede--Daykin--Schur inequality developed in  \cite[Theorem 1.6]{CCPS}. Intuitively, to acquire  a proof of Theorem \ref{main-1}, we need  a `multivariate' version of the  Ahlswede--Daykin--Schur inequality.  However, we do not find an appropriate multivariate extension   that is applicable to Theorem \ref{main-1}. 
\end{itemize}
In order to address  the above issues, we require substantial  machinery. 

\begin{itemize}
    \item We establish   a new formula for the coefficients $K_{\lambda,\mu}(\mathbf{t},\mathbf{w})$ in terms of oscillating sequences, see Corollary \ref{ciri-j0}. The  idea behind this is to incorporate the operations of  Bandlow and Morse \cite{Bandlow} where the RSK and jeu de taquin algorithms are used to   interpret the coefficients  in the Schur expansion of stable and dual stable Grothendieck polynomials. The advantage is that oscillating sequences naturally form a distributive lattice, see Corollary  \ref{cln-sl}, so that we could  be able to  use  the framework  that have been developed  by Chan, Chen, Pak and Soskin \cite{CCPS} to deal with correlation inequalities. 

    \item We use  a multivariate  Reuter--Lov\'asz--Saks type inequality on finite distributive lattices,  see Theorem  \ref{WRLS}, as well as  
    a variation of the Schur orchestra   inequality    by Chan, Chen, Pak and Soskin \cite[Theorem 5.2]{CCPS}, see  Theorem \ref{cor:straight-complementary-orchestra}. 
\end{itemize}

\begin{remark}
In private communications,  Swee Hong Chan pointed out that Theorem   \ref{WRLS} can be implied (not immediately obviously) by \cite[Claim 6.3]{CP23}, and  Theorem \ref{cor:straight-complementary-orchestra}  follows from  \cite[Theorem 7.14]{CCPS}. Since the statement in \cite[Theorem 7.14]{CCPS} is not needed in \cite{CCPS},  its proof is  omitted in \cite{CCPS}. For clarity and   self-containedness,   we will include  proofs of Theorems  \ref{WRLS}  and \ref{cor:straight-complementary-orchestra} in this paper.  
\end{remark}

\begin{remark}
 The Okounkov   inequality for Schur functions was conjectured by Okounkov \cite{Okounkov}, and proved by Lam, Postnikov and Pylyavskyy \cite[Theorem 4]{LamPostnikovPylyavskyy2007}: 
\[
   s_\lambda(\mathbf x)s_\mu(\mathbf x)
   \leq_s
   s_{\left\lceil(\lambda+\mu)/2\right\rceil}(\mathbf x)s_{\left\lfloor(\lambda+\mu)/2\right\rfloor}(\mathbf x).
\]
Combining Theorems \ref{thm-ek} and \ref{thm-ek-d} with Speyer's theorem \cite{Speyer2026} (see \cite[Theorem 3.5]{CCPS}),   Chan, Chen, Pak and Soskin \cite[Corollaries 9.5 and 10.3]{CCPS}  obtained the Okounkov   inequalities  for stable and dual stable Grothendieck polynomials.
We can also prove a   multivariate  Okounkov  inequality for hybrid Grothendieck polynomials. 
The proof needs additional techniques, and  will be presented in future work.     
\end{remark}

\subsection{Conjectured extensions of the LPP inequality}

In the Schubert calculus of the flag variety, Schur polynomials arise canonically as a subclass of Schubert polynomials, see Remark \ref{re-mno0-1}. 
It is therefore natural  to ask if there exists  a Schubert positivity extension   which    includes
the classical Schur LPP inequality     as a special case.  
Let $S_n$ denote the symmetric group of permutations of $\{1,2,\ldots, n\}$. Write  
$w=w(1)\cdots w(n)\in S_n$  in one-line notation. 
The
Schubert polynomials  $\mathfrak{S}_w(\mathbf x)$,   introduced by   Lascoux  and   Schützenberger \cite{lascoux1989fonctorialite}, represent   the Schubert classes   in  the  cohomology ring of the    flag variety.

To define  $\{\mathfrak{S}_w(\mathbf x)\colon  w\in S_n\}$, recall   the  divided difference operator  $\partial_i$ acting on  a polynomial $f(\mathbf{x})$ by letting  $\partial_i(f)=\frac{f- s_i\cdot f}{x_i-x_{i+1}},$ 
where $s_i\cdot f$ is obtained from $f$ by interchanging  the variables $x_i$ and $x_{i+1}$.
For $w_0=n\cdots 2 1$, set $\mathfrak{S}_{w_0} (\mathbf x)=x_1^{n-1}x_2^{n-2}\cdots x_{n-1}$.  
If $w\neq w_0$, then there must exist an index $1\leq i<n$ such that $w(i)<w(i+1)$.  Define $\mathfrak{S}_{w} (\mathbf x)=\partial_i \,\mathfrak{S}_{ws_i}(\mathbf x),$  where $ws_i$ is obtained from $w$ by swapping $w(i)$ and $w(i+1)$.   
This is independent of the choice of $i$ since the operators $\partial_i$ satisfy the Coxeter  relations.

The Schubert polynomials   $\mathfrak{S}_w(\mathbf x)$ where $w\in S_n$ are linearly independent, constituting  a basis 
of the space $\mathrm{span}_\mathbb R\{x_1^{a_1}\cdots x_n^{a_n}: 0\leq a_i\leq n-i\}$.  For
two polynomials $f$ and $g$,   write $f\leq_{\mathfrak S} g$
  if  $g-f$ is Schubert positive:
\[
g-f\in \sum_{w} \mathbb R_{\geq 0} \cdot \mathfrak S_w.
\]
An expected  Schubert LPP inequality would look  like 
\begin{equation}\label{eq-jo}
    \mathfrak S_u(\mathbf x)\mathfrak S_v(\mathbf x)
\leq_{\mathfrak S}
\mathfrak S_{u\vee v}(\mathbf x)\mathfrak S_{u\wedge v}(\mathbf x).
\end{equation}
The first task  is how  to define the  join and meet   of two permutations. To do this, we use the   Lehmer code (also called  inversion code) of $w$:
\[
c(w)=(c_1(w),\ldots,c_n(w)),\qquad
c_i(w)=\#\{j:i<j\leq n,\ w(i)>w(j)\}.
\]
It is easily seen that  $0\leq c_i(w)\leq n-i$, and conversely  every such integer vector determines
a   permutation in $S_n$.  
For $u,v\in S_n$, define their  join and
meet are permutations determined  by the following Lehmer codes:
\[
c_i(u\vee v)=\max\{c_i(u),c_i(v)\},\qquad
c_i(u\wedge v)=\min\{c_i(u),c_i(v)\}.
\]
Obviously, the above  are indeed valid Lehmer codes of length $n$, and so   
$u\vee v$ and $u\wedge v$ are permutations belonging to $S_n$.

\begin{remark}\label{re-mno0-1}
Let $w$ be a   $k$-Grassmannian, that is, it has at most one descent at position $k$. It is easy to check that 
$c_1(w) \leq \cdots\leq c_k(w)$ and $c_i(w)=0$ for $i>k$. Denote by $\lambda(w)=(c_k(w),\ldots, c_1(w))$   the associated partition. It is a classical  fact that
$\mathfrak S_w(\mathbf x)$ is the Schur polynomial $s_{\lambda(w)}(x_1,\ldots, x_k)$.  Notice that when   $u$ and $v$ are $k$-Grassmannians,  both $u\vee v$ and $u\wedge v$ are also $k$-Grassmannians, and moreover $\lambda(u\vee v)=\lambda(u)\vee \lambda(v)$ and $\lambda(u\wedge v)=\lambda(u)\wedge \lambda(v)$. Therefore, in such a case,  \eqref{eq-jo} will reduce to the LPP inequality in Theorem  \ref{llp-kk}.
\end{remark}

For general $u$ and $v$, we could immediately find examples that violate  \eqref{eq-jo}:  for instance $u=132$ and $v=213$. A reasonable route  is to  consider permutations whose Schubert polynomials are close to Schur polynomials. Arguably, vexillary permutations are such permutations: $w$ is {vexillary} if it avoids the pattern
$2143$, that is, there do not exist $i<j<k<\ell$ such that $w(j)<w(i)<w(\ell)<w(k)$. This class of permutations plays an important role in the combinatorics and geometry of  Schubert varieties and degeneracy loci,  see for example  Fulton \cite{Fulton-1},  Eriksson--Linusson \cite{EL} and  Knutson--Miller--Yong \cite{KMY-1}. 

Wachs' theorem \cite{Wachs1985} identifies the Schubert polynomial of a
vexillary permutation with a flagged Schur polynomial.  
Precisely, for $w\in S_n$, set
\[
I_i(w)=\{j:i<j\leq n,\ w(i)>w(j)\}.
\]
So we have $c_i(w)=\# I_i(w)$.
For every $i$ with $c_i(w)>0$, define
\[
\beta_i(w)=\min I_i(w)-1.
\]
Arrange the entries $\{c_i(w)>0\}$ in weakly decreasing order to get a partition $\lambda(w)=(\lambda_1\geq \cdots\geq \lambda_k)$, and the values $\{\beta_i(w): c_i(w)>0\}$ in weakly increasing order to get a flag $\phi(w)=(\phi_1\leq \cdots\leq  \phi_k)$.  
  When  $w$ is vexillary, it is well known that  $\mathfrak S_w=s_{\lambda(w)}^{\phi(w)}$,
the flagged Schur polynomial of shape $\lambda(w)$ and flagged by  $\phi(w)$   \cite{Wachs1985}. 

\begin{remark}
A Schubert polynomial $\mathfrak S_w$ is a flagged Schur polynomial if and only if $w$ is vexillary \cite{lascoux1989fonctorialite}. However,   not every flagged Schur polynomial can be realized as a Schubert polynomial.
\end{remark}

For two finite sequences $a=(a_1,\ldots,a_r)$ and
$b=(b_1,\ldots,b_s)$, write $a\preceq b$
if $a$ is a prefix of $b$, that is,   $r\leq s$ and $a_i=b_i$ for
$1\leq i\leq r$.

\begin{conjecture}[vexillary Schubert LPP inequality]\label{con-1jefi}
Let $u,v\in S_n$ be vexillary permutations.  Suppose that their Wachs
flags satisfy $\phi(u)\preceq \phi(v)$. 
Then
\[
\mathfrak S_u(\mathbf x) \mathfrak S_v(\mathbf x) 
\leq_{\mathfrak S}
\mathfrak S_{u\vee v}(\mathbf x) \mathfrak S_{u\wedge v}(\mathbf x).
\]
\end{conjecture}

The above conjecture has been verified  for $n\leq 8$. 

\begin{remark}
Notice that Grassmannian permutations are vexillary permutations. The Wachs flag of a $k$-Grassmannian has all entries equal to $k$. Hence when  both $u$ and $v$ are  $k$-Grassmannians,     Conjecture \ref{con-1jefi} would  specialize to Theorem  \ref{llp-kk}. 
\end{remark}

We give an example to illustrate  Conjecture \ref{con-1jefi}. 
Take two vexillary permutations $u=5712364$ and $ v=3514672$ in $S_7$, which have Lehmer codes
\[
c(u)=(4,5,0,0,0,1,0),\ \ \ \ \ c(v)=(2,3,0,1,1,1,0),
\]
and Wachs flags
\[
\phi(u)=(2,2,6)\,\preceq\,\phi(v)=(2,2,6,6,6).
\]
So we have  
$u\vee v=5713462$ and  $u\wedge v=3512476$.
Note that $u\wedge v$  has the pattern  $2143$, and thus is no longer  vexillary.  By  direct computation,  
\[
\mathfrak S_{u\vee v} \mathfrak S_{u\wedge v}-\mathfrak S_u \mathfrak S_v =2\, \mathfrak S_{(7,8,0,1,1,1)}+\mathfrak S_{(6,9,0,1,1,1)},
\]
where, on  the right-hand side,  we use  the Lehmer codes to indicate  permutations (the former  is a permutation in $S_{10}$, and the latter is in $S_{11}$).

We also formulate an equivariant version of Conjecture \ref{con-1jefi}. The double Schubert polynomials $\mathfrak S_w(\mathbf x; \mathbf y)$ represent the Schubert classes in the torus-equivariant cohomology of the flag variety. In the equivariant case, we set  
\[
\mathfrak{S}_{w_0} (\mathbf x;\mathbf y)=\prod_{i+j\leq n}(x_i-y_j).
\]
If $w\neq w_0$, let  $\mathfrak{S}_{w} (\mathbf x; \mathbf y)=\partial_i \,\mathfrak{S}_{ws_i}(\mathbf x; \mathbf y)$ where $w(i)<w(i+1)$. Here $\partial_i$ acts on the $\mathbf x$ variables. 
It is clear that $\mathfrak{S}_{w} (\mathbf x)=\mathfrak{S}_{w} (\mathbf x; \mathbf y)|_{y_i=0}$.

For $i\geq 1$, write 
$\alpha_i=y_{i+1}-y_i$. By Graham's positivity theorem \cite{Gramhan}, the structure constants $c^w_{u,v}(\mathbf y) $ in the expansion of  $\mathfrak S_u(\mathbf x; \mathbf y)\mathfrak S_v(\mathbf x; \mathbf y)$ satisfies 
\[
\mathfrak S_u(\mathbf x; \mathbf y)\mathfrak S_v(\mathbf x; \mathbf y)=\sum_{w} c^w_{u,v}(\mathbf y) \mathfrak S_w(\mathbf x; \mathbf y),\ \ \text{where $c^w_{u,v}(\mathbf y)\in \mathbb Z_{\geq 0}[\alpha_1,\alpha_2,\ldots]$}.
\]
This motivates  the equivariant Schubert positivity:   write $f\leq_{\mathfrak S}^{\text{equiv}} g$
  if   
\[
g-f\in \sum_{w} \mathbb R_{\geq 0}[\alpha_1,\alpha_2,\ldots] \cdot \mathfrak S_w(\mathbf x; \mathbf y).
\]

\begin{conjecture}[vexillary equivariant Schubert LPP inequality]\label{con-1jefi-equi}
Let $u,v\in S_n$ be vexillary permutations.  Suppose that their Wachs
flags satisfy $\phi(u)\preceq \phi(v)$. 
Then
\[
\mathfrak S_u(\mathbf x; \mathbf y) \mathfrak S_v(\mathbf x; \mathbf y) 
\leq_{\mathfrak S}^{\mathrm{equiv}}
\mathfrak S_{u\vee v}(\mathbf x; \mathbf y) \mathfrak S_{u\wedge v}(\mathbf x; \mathbf y).
\]
This conjecture  will  specialize to Conjecture \ref{con-1jefi}  by taking $\mathbf y=\mathbf 0$.
\end{conjecture}

Because of the computational complexity of double Schubert polynomials, we tested Conjecture \ref{con-1jefi-equi} randomly for vexillary permutations in $S_n$ for $n=9$. 

When  restricting  $w$ to a $k$-Grassmannian, $\mathfrak S_w(\mathbf x; \mathbf y)$ is equal to the double Schur polynomial $s_{\lambda(w)}(\mathbf x_k; \mathbf y)$, where $\mathbf x_k=(x_1,\ldots, x_k)$,    which is also called the factorial Schur polynomial, see for example \cite{KT,MS}. In the Grassmannian case, Conjecture \ref{con-1jefi-equi} leads to an equivariant extension of Theorem  \ref{llp-kk}.

\begin{conjecture}[equivariant  Schur LPP inequality]\label{equi-nlvnsln-12}
For   partitions $\lambda$ and $\mu$, we have 
\[
   s_\lambda(\mathbf x_k; \mathbf y)s_\mu(\mathbf x_k; \mathbf y)
   \leq_s^{\mathrm{{equiv}}}
   s_{\lambda\vee\mu}(\mathbf x_k; \mathbf y)s_{\lambda\wedge\mu}(\mathbf x_k; \mathbf y).
\]
This conjecture will   specialize  to the classical  LPP inequality  by taking $\mathbf y=\mathbf 0$.
\end{conjecture}

Conjecture \ref{equi-nlvnsln-12} has been checked 
for all  Grassmannians in $S_n$ with $n$ up to $7$, and randomly for Grassmannians in $S_n$ with $n=10$. 

We lastly turn to the  Grothendieck setting. As the  K-theory counterpart,     Grothendieck polynomials $\mathfrak{G}_w(\mathbf x)$ are   polynomial representatives of   Schubert classes in the Grothendieck ring of the    flag variety \cite{LS-Gro}.  We shall use  the $\beta$-version $\mathfrak{G}_w^{(\beta)}(\mathbf  x)$ defined by  Fomin and Kirillov   \cite{fomin1994grothendieck}, which specializes to $\mathfrak{S}_w(\mathbf x)$ and $\mathfrak{G}_w(\mathbf x)$ by taking $\beta=0$ and $\beta=-1$, respectively. 

The $\beta$-Grothendieck polynomials  are  defined in the same way as Schubert polynomials, except that we replace $\partial_i$ by the   operator 
$\pi_i^{(\beta)}(f)=
\partial_i((1+\beta x_{i+1})f)$.
Now, for two polynomials $f$ and $g$, write $f\leq_{\mathfrak{G}^{(\beta)}} g$
if
\[
g-f\in \sum_w \mathbb R_{\geq 0}[\beta]\cdot  \mathfrak{G}_w^{(\beta)}.
\]

The following  K-theory analogue  of Conjecture \ref{con-1jefi}  has been checked for $n\leq 7$. 
 
\begin{conjecture}[vexillary Grothendieck LPP inequality]\label{con-1jefi-G}
Let $u,v\in S_n$ be vexillary permutations.  Suppose that their Wachs
flags satisfy $\phi(u)\preceq \phi(v)$. 
Then
\[
\mathfrak{G}_u^{(\beta)}(\mathbf x)\mathfrak{G}_v^{(\beta)}(\mathbf x)
\leq_{\mathfrak{G}^{(\beta)}}
\mathfrak{G}^{(\beta)}_{u\vee v}(\mathbf x)\mathfrak{G}^{(\beta)}_{u\wedge v}(\mathbf x).
\]
This conjecture will specialize to Conjecture \ref{con-1jefi} by taking $\beta=0$. 
\end{conjecture}

Similarly to the Schubert case, when $w$ is a vexillay permutation, $\mathfrak{G}_w^{(\beta)}(\mathbf x)$ can be constructed  by flagged set-valued tableaux, see for example \cite{KMY-1}. 
In particular, for $w$
a $k$-Grassmannian, its Grothendieck polynomial (for $\beta=1$) is equal to    $\widetilde{G}_{\lambda(w)}(\mathbf x_k)$. Hence when restricted to $k$-Grassmannians, Conjecture \ref{con-1jefi-G} reduces to   a statement for stable Grothendieck polynomials, as previously predicted   by Thomas and Yong \cite[Conjecture 9.2]{TY-1} (see also   Mihalcea \cite{Mihalcea}) which remains open. 

\begin{conjecture}[\text{\cite[Conjecture 9.2]{TY-1}}]\label{dsj-0sdj-09} 
For   partitions $\lambda$ and $\mu$, we have 
\[
   \widetilde{G}_\lambda(\mathbf x)\widetilde{G}_\mu(\mathbf x)
\leq_{\widetilde{G}}
   \widetilde{G}_{\lambda\vee\mu}(\mathbf x)\widetilde{G}_{\lambda\wedge\mu}(\mathbf x).
\]
\end{conjecture}

The double Grothendieck polynomials are the K-theory analogue of double Schubert polynomials. The structure constants in this setting are positive in the sense of Anderson, Griffeth and Miller \cite{AGM}. The numerical data indicate that Conjectures \ref{con-1jefi-G} and  \ref{dsj-0sdj-09} might also possess an equivariant Grothendieck LPP positivity. 

The structure of this paper is  as follows.  In Section \ref{sec:preliminaries}, we lay out necessary definitions and notation for this paper. In Section \ref{sec-s33}, we introduce  the notion of oscillating sequences which is the starting point of our proof for Theorem \ref{main-1}. In particular, oscillating sequences form a distributive lattice. 
Sections \ref{sec:lpp} and   \ref{vari-srche}  are devoted to  a multivariate  Reuter--Lov\'asz--Saks type inequality and   the  {Schur orchestra inequality}, respectively.
Collecting  the results in  Sections \ref{sec-s33},    \ref{sec:lpp} and \ref{vari-srche}, we complete the proof of Theorem \ref{main-1} in Section \ref{sec:schur-positive}.

\subsection*{Acknowledgements}
We thank Swee Hong Chan for providing very valuable   comments on a draft  of this paper. 
This work was  supported by the National Natural Science Foundation of China (No. 12371329) and the Fundamental Research Funds for the Central Universities (No. 63263094).

\section{Preliminaries}\label{sec:preliminaries}

In this section, we shall collect some terminology and notation. Some basic  properties that we require  will also be mentioned. 
In particular, we refer to \cite[\S 1.1]{Fulton1997} and \cite[Chapter 7]{Stanley1999} for the theory of symmetric functions.

\subsection{Partition and Young lattice}

A partition is a weakly decreasing sequence $ \lambda=(\lambda_1,\lambda_2,\ldots) $ of nonnegative integers with only finitely many nonzero parts. Denote by   $|\lambda|=\sum_i \lambda_i$   the size of $\lambda$, and  by $\ell(\lambda)$  the number of nonzero parts of $\lambda$.
The Young diagram of   $\lambda$ is a left-justified array of boxes with $\lambda_i$ boxes in row $i$.  Let $(i,j)$   represent  the box in row $i$ and column $j$. Here the row (resp., column) indices are counted from top to bottom (resp., from left to right). Let $\mathcal{P}$ represent the set of all partitions.  We use $\mu \subseteq \lambda$ to mean that the Young diagram of $\mu$ is contained in that  of $\lambda$, namely, $\lambda_i\geq \mu_i$ for all $i\geq 1$.  Under the inclusion $\subseteq$ of partitions, $\mathcal{P}$ forms a distributive poset, called the Young lattice. Specifically, for two partitions $\lambda$ and $\mu$, their \textit{join} $\lambda\vee\mu$ and  \textit{meet} $\lambda\wedge\mu$ are partitions satisfying  
\[
(\lambda\vee\mu)_i=\max\{\lambda_i,\mu_i\},
\qquad
(\lambda\wedge\mu)_i=\min\{\lambda_i,\mu_i\}.
\]

For $\mu \subseteq \lambda$, the skew diagram $\lambda/\mu$ is obtained from the   diagram $\lambda$ by removing the boxes in $\mu$. We say  $\lambda/\mu$ is a \textit{horizontal} (resp., {\it vertical}) {\it strip} if  there are no two boxes lying in the same column (resp., row).
Here we allow a horizontal or vertical strip to be empty (namely, in the case $\lambda=\mu$).

\subsection{ Schur and flagged Schur function}

Let $\lambda/\mu$ be a skew shape.  A \textit{semistandard Young  tableau} of shape $\lambda/\mu$ is a filling 
$T\colon \lambda/\mu\longrightarrow \mathbb Z_{>0}$ of the boxes of $\lambda/\mu$ with positive integers, 
 such that the entries are weakly increasing from left to right along each row and strictly increasing from top to bottom along each column. In other words, whenever the corresponding boxes belong to $\lambda/\mu$, we have \[ T(i,j)\le T(i,j+1), \qquad T(i,j)<T(i+1,j). \]

Denote by $\SSYT(\lambda/\mu)$ the set of all semistandard  tableaux of shape $\lambda/\mu$. For   $T\in \SSYT(\lambda/\mu)$, its weight is the sequence $\wt(T)=(\wt_1(T),\wt_2(T),\ldots),$  where $\wt_i(T)$ is the number of entries of $T$ equal to $i$. Write \[ \mathbf{x}^{\wt(T)} = \prod_{i\ge 1} x_i^{\wt_i(T)}. \] 
The   \textit{Schur function} indexed by skew shape $\lambda/\mu$ is defined by \[ s_{\lambda/\mu}(\mathbf{x}) = \sum_{T\in \SSYT(\lambda/\mu)} \mathbf{x}^{\wt(T)}. \] 
In the straight shape case when $\mu=\emptyset$, we write $s_\lambda(\mathbf{x})$ instead of $s_{\lambda/\emptyset}(\mathbf{x})$. If we use a finite number of variables $\mathbf x_n=(x_1,\ldots, x_n)$, then we are given   a Schur polynomial. 
It is a classical result that the set of  Schur functions forms a linear basis of the ring of symmetric functions.

The flagged Schur polynomials are generalizations of Schur polynomials. A \textit{flag} means 
a weakly increasing sequence  $\phi=(\phi_1,\phi_2,\ldots,\phi_r)$  of positive integers. 
Suppose that $\lambda/\mu$ has at most $r$ rows. A \textit{flagged semistandard tableau }of shape $\lambda/\mu$ and with flag $\phi$ is a semistandard tableau   of shape $\lambda/\mu$ such that every entry in row $i$ cannot exceed  $\phi_i$.
Denote the set of such tableaux by  $\SSYT(\lambda/\mu, {\phi})$. The corresponding \textit{flagged Schur function} is \[ s_{\lambda/\mu}^{\phi}(\mathbf{x}) = \sum_{T\in \SSYT(\lambda/\mu, {\phi})} \mathbf{x}^{\wt(T)}. \]

\subsection{Stable and dual stable Grothendieck polynomial}

Stable Grothendieck polynomials $G_\lambda(\mathbf{x})$  and their duals $g_\lambda(\mathbf{x})$ are nonhomogeneous generalizations of Schur functions. 
We first recall the set-valued tableau construction of $G_\lambda(\mathbf{x})$ pioneered by Buch \cite{Buch2002}. For two nonempty finite sets $A,B\subset \mathbb Z_{>0}$, we write $A\le B$ (resp., $A<B $) to mean that $\max A\le \min B$  (resp., $\max A< \min B$). 
A \textit{set-valued tableau} of shape $\lambda/\mu$ is a filling $T$ of each box of $\lambda/\mu$ with a nonempty finite subset of $\mathbb Z_{>0}$ such that the sets are weakly increasing along each row, and strictly increasing down each column. 
Let $\SVT(\lambda/\mu)$ be the set of all   set-valued tableaux of shape $\lambda/\mu$.

For $T\in \SVT(\lambda/\mu)$, define $|T|=\sum_{(i,j)\in \lambda/\mu} \# T(i,j)$.  Its weight is $\wt(T) = (\wt_1(T),\wt_2(T),\ldots),$   where $\wt_i(T)$ is the total number of occurrences of $i$  appearing in $T$. The \textit{stable Grothendieck polynomial} of  skew shape $\lambda/\mu$ is defined by \[ G_{\lambda/\mu}(\mathbf{x}) = \sum_{T\in \SVT(\lambda/\mu)} (-1)^{|T|-|\lambda/\mu|} \mathbf{x}^{\wt(T)}. \] 
It is easy to see that  the Schur function $s_\lambda$  is equal to the lowest-degree homogeneous component   of $G_{\lambda}(\mathbf{x})$. 

The Hall inner product on symmetric functions is defined by setting $<s_\lambda, s_\mu>=\delta_{\lambda, \mu}$. As the dual basis of  $G_{\lambda}(\mathbf{x})$, 
the   \textit{dual stable Grothendieck polynomials} $g_{\lambda}(\mathbf{x})$  were extensively  studied by  Lam and Pylyavskyy \cite{LamPylyavskyy2007}.  Unlike $G_{\lambda}(\mathbf{x})$, the dual  $g_{\lambda}(\mathbf{x})$ is combinatorially constructed in terms of reverse plane partitions. Recall that a \textit{reverse plane partition} of shape $\lambda/\mu$ is a filling  $T\colon \lambda/\mu\longrightarrow \mathbb Z_{>0}$   of the boxes in $\lambda/\mu$   such that the entries are weakly increasing along each row and each column. Denote  by  $\RPP(\lambda/\mu)$ the set of reverse plane partitions of shape $\lambda/\mu$.  

For $T\in \RPP(\lambda/\mu)$, the {irredundant content} vector $\ircont(T)=(r_1,r_2,\ldots)$ is defined by letting  $r_i$ be the number of columns  contains at least one $i$. Define  \[ g_{\lambda/\mu}(\mathbf{x}) = \sum_{T\in \RPP(\lambda/\mu)} \mathbf{x}^{\ircont(T)}. \]
It follows from the definition that the highest-degree homogeneous component   of $g_{\lambda}(\mathbf{x})$ is the Schur function $s_\lambda$.

 Lam and Pylyavskyy \cite{LamPylyavskyy2007} showed that $<G_\lambda, g_\mu>=\delta_{\lambda, \mu}$.  Importantly, this relation yields an alternative way to compute the structure constants in the K-theory of Grassmannians, see also Li,  Morse and   Shields \cite{LMS}. 

 The expansions of $G_\lambda$ and $g_\lambda$ in  the basis of Schur functions have been given in \cite{Buch2002} and \cite{LamPylyavskyy2007},  respectively,  see also Bandlow and Morse \cite{Bandlow}. The coefficients appearing in the expansion have explicit combinatorial formulas in terms of  certain kinds of tableaux. Such  formulas are crucial  in the proofs of  Theorems \ref{thm-ek} and \ref{thm-ek-d}, as given by 
Chan, Chen, Pak and Soskin \cite{CCPS}.

\subsection{Hybrid Grothendieck polynomial}

The hybrid Grothendieck polynomials $ H_{\lambda}(\mathbf{x};\mathbf{t},\mathbf{w})$, which are defined based on set-valued reverse plane  partitions,  provide   a unified platform to investigate stable and dual stable Grothendieck polynomials. 
A  \textit{set-valued reverse plane partition} of shape $\lambda/\mu$ is a filling of every box of $\lambda/\mu$ with  a nonempty finite set of positive integers,  such that the sets are weakly increasing along each row and each column. See Figure \ref{fig-agk-90}  for an illustration. The set of such fillings is denoted  by $\SVRPP(\lambda/\mu)$.
\begin{figure}[h t]
\begin{center}
\begin{tabular}{c c c c c c c}
\begin{ytableau}
\none&\none&\none&2\\
\none&1 & 12 &23\\
12  & 23 & 3 &34\\
23 & 345 & 5\\
35
\end{ytableau}
\end{tabular}
\end{center}
\vspace{-5pt}
\caption{ A set-valued  reverse plane  
partition. Here we have omitted  the braces $\{\ \}$ for sets as well as  the commas between numbers.}
\label{fig-agk-90}
\end{figure}

For $T\in\SVRPP(\lambda/\mu)$ and a box $(i,j)\in \lambda/\mu$, let $T(i,j)$ be the set in $(i,j)$. To define $ H_{\lambda}(\mathbf{x};\mathbf{t},\mathbf{w})$, we need  the following   statistics:
\begin{itemize}
\item the \emph{irredundant content vector} $\ircont(T)=(r_1,r_2,\ldots)$, where $r_i$ is the number of columns containing a set which contains $i$;

\item the \emph{column equality vector} $\ceq(T)=(c_1,c_2,\ldots)$, where $c_i$ is the number of boxes $(i,j)$ in row $i$ such that $(i+1,j)\in\lambda/\mu$ and $\max T(i,j)=\min T(i+1,j)$;

\item the \emph{excess vector} $\ex(T)=(e_1,e_2,\ldots)$, where
\[
   e_i=
   \sum_{(i,j)\in\lambda/\mu}\bigl(\# T(i,j)-1\bigr).
\]
\end{itemize}
For the set-valued reverse plane partition in  Figure \ref{fig-agk-90}, we have 
\begin{align*}
\mathrm{ircont}(T)=(3,4,4,2,3), \ \ \  \mathrm{ceq} (T)=(1,1,2,1),   \ \ \  \mathrm{ex} (T)=(0,2,3,3,1).
\end{align*}

The \textit{hybrid Grothendieck polynomial} of shape $\lambda/\mu$ is
\begin{equation*}\label{eq:Hp-definition}
  H_{\lambda/\mu}(\mathbf{x};\mathbf{t},\mathbf{w})
  =
  \sum_{T\in\SVRPP(\lambda/\mu)}
  \mathbf{x}^{\ircont(T)}\mathbf{t}^{\ceq(T)}\mathbf{w}^{\ex(T)}.
\end{equation*}
We emphasize that  $\mathbf{x}=(x_1,x_2,\ldots)$, $\mathbf{t}=(t_1,t_2,\ldots)$  and $\mathbf{w}=(w_1,w_2,\ldots)$ are countable sequences of variables, and $\mathbf y^\alpha=y_1^{\alpha_1}y_2^{\alpha_2}\cdots$ for a weak composition $\alpha$.

As aforementioned in  Remark \ref{rmk-agh}, hybrid Grothendieck polynomials unify (refined) stable and dual stable  Grothendieck polynomials.  If setting all  $t_i=0$, then     ${H}_{\lambda}(\mathbf{x};\mathbf{0},\mathbf{w})=G_\lambda(\mathbf x; \mathbf w)$ is   the \textit{refined stable Grothendieck polynomial}  due to Chan  and Pflueger \cite{ChanPflueger2021}. If  further setting  all  $w_i=-1$, then ${H}_{\lambda}(\mathbf{x};\mathbf{0},\mathbf{-1})$ is    $G_\lambda(\mathbf{x})$. We also have  ${H}_{\lambda}(\mathbf{x};\mathbf{0},\mathbf{1})=\widetilde{G}_\lambda(\mathbf x)$.
On the other hand, if setting all $w_i=0$, then      ${H}_{\lambda}(\mathbf{x};\mathbf{t},\mathbf{0})=g_\lambda(\mathbf x;\mathbf t)$ is  the \textit{refined dual stable Grothendieck polynomial}  introduced   by Galashin, Grinberg and Liu \cite{GalashinGrinbergLiu2016}. If further setting all   $t_i=1$, then ${H}_{\lambda}(\mathbf{x};\mathbf{1},\mathbf{0})$ becomes  $g_\lambda(\mathbf{x})$.

It was shown in  \cite[Theorem 1.2]{GuoKangLiu2025}
that $H_{\lambda/\mu}(\mathbf{x};\mathbf{t},\mathbf{w})$ is indeed symmetric in $\mathbf x$. Moreover,  $H_{\lambda/\mu}(\mathbf{x};\mathbf{t},\mathbf{w})$ is Schur positive. Using the crystal basis method, the coefficients  $K_{\lambda,\mu}(\mathbf{t},\mathbf{w})$  in the Schur expansion 
\begin{equation} \label{ck-dk}
H_{\lambda}(\mathbf{x};\mathbf{t},\mathbf{w})
   =
   \sum_\mu K_{\lambda,\mu}(\mathbf{t},\mathbf{w})s_\mu(\mathbf{x})
\end{equation}
can be formulated in terms of set-valued reverse  plane partitions satisfying the lattice word condition \cite[Theorem 1.3]{GuoKangLiu2025}. Note that \cite[Theorem 1.3]{GuoKangLiu2025} is valid for general  skew shapes. This paper  only concerns   the straight shape situation. 

For our purpose to prove  Theorem \ref{main-1}, the formula of 
$K_{\lambda,\mu}(\mathbf{t},\mathbf{w})$ given in \cite[Theorem 1.3]{GuoKangLiu2025} is not available, because of the absence of a distributive lattice structure.  To address this problem, inspired by the work  of Bandlow and Morse  \cite{Bandlow}, we design a  dilation--contraction algorithm on set-valued reverse plane partitions to give a new formula for $K_{\lambda,\mu}(\mathbf{t},\mathbf{w})$, see  Section \ref{sec-s33}.

\subsection{RSK algorithm}\label{subsec:rsk-row-insertion}

The \textit{RSK algorithm} inserts a positive integer $x$ to a semistandard tableau  $T$ of straight shape, yielding a new semistandard tableau. 
The algorithm 
is described  in a recursive procedure. 
First, insert $x$ into the first row of $T$. If all entries in the first row are
less than or equal to $x$, then place $x$ in a new box at the end of first  row
and stop. Otherwise, let $y$ be the smallest  entry in the first row satisfying
$y>x$. Replace  $y$ by $x$ and insert the bumped entry $y$ into the
second row. Repeating the same procedure,  the row insertion algorithm eventually terminates, giving rise to a new semistandard tableau, denoted 
$T\leftarrow x$. Note that the shape of $T\leftarrow x$ has one more extra corner, denoted $B$,  than the shape of $T$. 
The  RSK algorithm is reversible in the sense that one  can uniquely recover $T$ and $x$ from $T\leftarrow x$ and the position of the  extra corner box $B$. 
The collection of boxes where an entry is bumped, together with the corner  box where the last bumped entry lands, is called the bumping route. 
Figure     \ref{fig-jpg} gives an illustration of the RSK insertion for $x=3$, where the boxes on the bumping route are  shaded. 
\begin{figure}[h t]
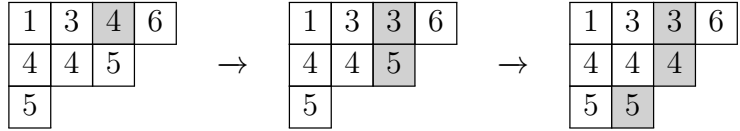

    \centering    
\[ \begin{aligned}
\ytableausetup{boxsize=1.3em}
\begin{ytableau}
1 & 3 & *(routegray)4 & 6 \\
4 & 4 & 5 \\
5 
\end{ytableau}
\quad
\xrightarrow{}
\quad
\begin{ytableau}
1 & 3 &*(routegray)3 & 6 \\
4 & 4 & *(routegray)5 \\
5
\end{ytableau}
\quad
\xrightarrow{}
\quad
\begin{ytableau}
1 & 3 & *(routegray)3 & 6 \\
4 & 4 &*(routegray)4 \\
5 & *(routegray)5
\end{ytableau}
\end{aligned}\]
\caption{An illustration of the RSK algorithm.}
    \label{fig-jpg}
\end{figure}

The following property   is not hard to justify,  see   \cite[Chapter 1]{Fulton1997}.

\begin{proposition}\label{rsk}
Let $T$ be a semistandard tableau of shape $\lambda$, and let
\[
U=(((T\leftarrow x_1)\leftarrow x_2)\cdots )\leftarrow x_p
\]
be obtained    by successively  inserting
$x_1,\ldots,x_p$. Let $\mu$ be the shape of $U$.
If $x_1>x_2>\cdots >x_p$, 
then $\mu/\lambda$ forms a vertical strip.
Conversely, suppose that $U$ is a semistandard tableau  of shape $\mu$, and that  $\lambda$ is a partition in $\mu$ such that $\mu/\lambda$ has $p$ boxes. If $\mu/\lambda$ is a vertical strip, then there exist a unique semistandard tableau  $T$ of shape $\lambda$ and a unique sequence $x_1>x_2>\cdots >x_p$
such that
\[
U=(((T\leftarrow x_1)\leftarrow x_2)\cdots )\leftarrow x_p.
\]
\end{proposition}

\subsection{Jeu de taquin algorithm}\label{subsec:jdt}

Let $\lambda/\mu$ be a skew diagram. An \emph{inside corner} of $\lambda/\mu$
is a box which may be removed from $\mu$ while still leaving a valid Young diagram, while an  \emph{outside corner} of $\lambda/\mu$ is a box which
may be removed from $\lambda$ while still leaving a Young diagram.  In Figure    \ref{fig-jpg-er}, the inside and outside corners are marked with $\circ$  and $\times$, respectively. 
\begin{figure}[h t]
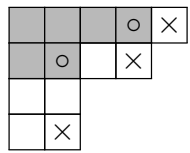

    \centering   
    \vspace{-10pt}
    \[
\ytableausetup{boxsize=1.1em}
\begin{ytableau}
*(gray!55) & *(gray!55) & *(gray!55) & *(gray!55){\circ} & {\times} \\
*(gray!55) &*(gray!55) {\circ} &  & {\times} \\
 &   \\
 & {\times}
\end{ytableau}
\]
    \vspace{-10pt}
\caption{Inside and outside corners.}
    \label{fig-jpg-er}
\end{figure}

Let $S$ be a  semistandard tableau  of shape $\lambda/\mu$. Suppose that  $C$ is an inside
corner of $\lambda/\mu$. Regard  $C$ as an empty box. The
\textit{jeu de taquin (jdt) slide} starting from $C$ is defined as follows. 
At each step, compare the entries in the two boxes immediately to the right
and immediately below the empty box.
Move the smaller of these two entries into the empty cell. If only one of the two
neighbors exists, move that entry into the empty cell. If both neighbors exist and
have equal entries, move the lower entry into the empty cell. The empty box has now moved
to the position of the entry just moved.
This process is repeated until the empty box reaches an outside corner, yielding a new   semistandard tableau.
The jdt slide is also reversible. Suppose that  $S'$ is the  semistandard tableau  obtained after a jdt slide, and 
that the outside corner $B$ removed at the end of the slide is specified. Replace $B$ by
an empty box. One can reverse the slide procedure  to recover the original tableau $S$. 
For example, in Figure \ref{fig-as-cf}, we apply the jdt algorithm to the inner box $(2,2)$ where the black boxes constitute the slide path.  

\begin{figure}[h t]
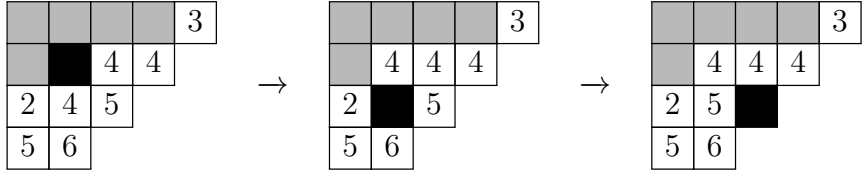

    \centering   
\[
\ytableausetup{boxsize=1.3em}
\setlength{\arraycolsep}{1.2em}
\begin{array}{ccccc}
\begin{ytableau}

\deletedbox & \deletedbox & \deletedbox & \deletedbox & 3 \\
\deletedbox & \jdtbox & 4 & 4 \\
2 & 4 & 5 \\
5 & 6
\end{ytableau}
\quad
\xrightarrow{}
\quad
\begin{ytableau}

\deletedbox & \deletedbox & \deletedbox & \deletedbox & 3 \\
\deletedbox & 4 & 4 & 4 \\
2 & \jdtbox & 5 \\
5 & 6
\end{ytableau}
\quad
\xrightarrow{}
\quad
\begin{ytableau}

\deletedbox & \deletedbox & \deletedbox & \deletedbox & 3 \\
\deletedbox & 4 & 4 & 4 \\
2 & 5 & \jdtbox \\
5 & 6
\end{ytableau}
\end{array}
\]
\caption{An illustration of the jdt slide.}
    \label{fig-as-cf}
\end{figure}


\section{Oscillating sequences}\label{sec-s33}

In this section, we introduce oscillating sequences to  give a new combinatorial formula for the coefficients $K_{\lambda,\mu}(\mathbf{t},\mathbf{w})$ appearing in the expansion in   \eqref{ck-dk}. The construction  is based on the RSK and jeu de taquin algorithms.  As a key property that we need,  
 oscillating sequences admit a distributive lattice structure.

\subsection{Oscillating sequences}

 Recall that a skew shape $\lambda/\mu$  is a horizontal (resp., vertical) strip if there are no two boxes lying  in the same column (resp., row). Here we allow a horizontal or vertical strip to be empty. 
A pair of partitions $(\lambda, \mu)$ is called a  {\textit{$k$-falling pair}} if 
\begin{enumerate}
    \item[(1)] $\lambda \supseteq \mu$ and $\lambda_i = \mu_i$ for all $ i \le k$,
    
    \item[(2)] $\lambda/\mu$ is a horizontal strip.
\end{enumerate}
Dually, we call a pair of partitions $(\lambda, \mu)$ a  {\textit{$k$-rising pair}} if
\begin{enumerate}
    \item[(1')] $\lambda \subseteq \mu$ and $\lambda_i = \mu_i$ for all $ i \le k$,
    
    \item[(2')] $\mu/\lambda$ is a vertical strip.
\end{enumerate}

\begin{definition}\label{def:oscillating-sequence}
Let $\lambda$ and $\mu$ be partitions. For $l\geq 1$, an  {$l$-oscillating sequence} from $\lambda$ to $\mu$ is a sequence $S=(\lambda^0,\lambda^1,\ldots,\lambda^{2l-1})$ of partitions
such that
\begin{enumerate}
    \item[(1)] $\lambda^0=\lambda$ and $\lambda^{2l-1}=\mu$;
    
    \item[(2)] For each $1\le i\le l$, the pair $(\lambda^{2i-2},\lambda^{2i-1})$ is an $(l-i+1)$-rising pair, and for each $1\le j\le l-1$, the pair $(\lambda^{2j-1},\lambda^{2j})$ is an $(l-j)$-falling pair.
\end{enumerate}
\end{definition}

\begin{remark}
In the remainder of this paper, we shall only be concerned with {$l$-oscillating sequences} from $\lambda$ to $\mu$ where  $l\ge \ell(\lambda)$.  
\end{remark}

\ytableausetup{boxsize=1.2em}
\newcommand{\ytb}{\phantom{\scriptstyle 1}}

For example, we take $\lambda=(4,3,2)$, $\mu=(4,3,3,1)$, and $l=3$.   Figure   \ref{fig-abfcf}
 gives an $3$-oscillating sequence from $\lambda$ to $\mu$.
\begin{figure}[h t]
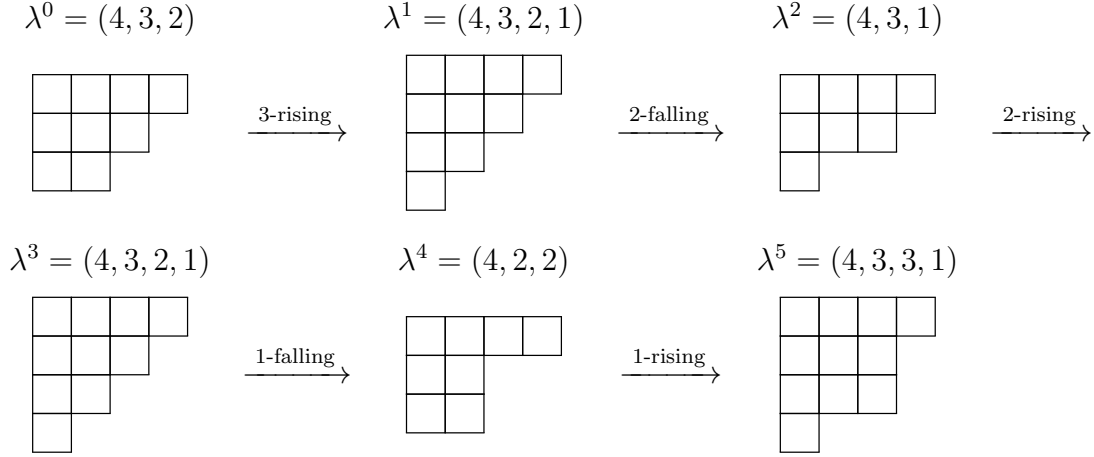

    \centering  
\[
\begin{array}{c@{\quad}c@{\quad}c@{\quad}c@{\quad}c@{\quad}c}
\lambda^0=(4,3,2)
&&
\lambda^1=(4,3,2,1)
&&
\lambda^2=(4,3,1)
&
\\[2mm]
\begin{ytableau}
\ytb & \ytb & \ytb & \ytb \\
\ytb & \ytb & \ytb \\
\ytb & \ytb
\end{ytableau}
&
\xrightarrow{\;3\text{-rising}\;}
&
\begin{ytableau}
\ytb & \ytb & \ytb & \ytb \\
\ytb & \ytb & \ytb \\
\ytb & \ytb \\
\ytb
\end{ytableau}
&
\xrightarrow{\;2\text{-falling}\;}
&
\begin{ytableau}
\ytb & \ytb & \ytb & \ytb \\
\ytb & \ytb & \ytb \\
\ytb
\end{ytableau}
&
\xrightarrow{\;2\text{-rising}\;}
\\[12mm]
\lambda^3=(4,3,2,1)
&&
\lambda^4=(4,2,2)
&&
\lambda^5=(4,3,3,1)
&
\\[2mm]
\begin{ytableau}
\ytb & \ytb & \ytb & \ytb \\
\ytb & \ytb & \ytb \\
\ytb & \ytb \\
\ytb
\end{ytableau}
&
\xrightarrow{\;1\text{-falling}\;}
&
\begin{ytableau}
\ytb & \ytb & \ytb & \ytb \\
\ytb & \ytb \\
\ytb & \ytb
\end{ytableau}
&
\xrightarrow{\;1\text{-rising}\;}
&
\begin{ytableau}
\ytb & \ytb & \ytb & \ytb \\
\ytb & \ytb & \ytb \\
\ytb & \ytb & \ytb\\
\ytb
\end{ytableau}
&
\end{array}
\]
\caption{An oscillating sequence from $(4,3,2)$ to $(4,3,3,1)$ with $l=3$.}
    \label{fig-abfcf}
\end{figure}

\begin{remark}
By definition, it is easy to see that for  an oscillating sequence   $S=(\lambda^0,\lambda^1,\ldots,\lambda^{2l-1})$, the first parts of $\lambda^i$ $(0\leq i\leq 2l-1)$ are of the same size.   
\end{remark}

Let $\mathrm{OS}^l(\lambda\rightarrow\mu)$ denote  the set of all  $l$-oscillating  sequences from $\lambda$ to $\mu$. Note that for fixed $\lambda$ and $\mu$, the set  $\mathrm{OS}^l(\lambda\rightarrow\mu)$ is finite  since there are finitely many choices for the added vertical strips or the deleted horizontal strips in the construction process.  
For an oscillating sequence  $S=(\lambda^0,\ldots,\lambda^{2l-1})$, define
\[
   \wt(S)=\prod_{1\le i\le l}w_{l-i+1}^{a_i}\prod _{1\leq j\leq l-1}t_{l-j}^{b_j},
   \qquad
   a_i=|\lambda^{2i-1}/\lambda^{2i-2}|,
   \quad
   b_j=|\lambda^{2j-1}/\lambda^{2j}|.
\]

We relate set-valued reverse plane partitions to oscillating sequences in the following theorem. 

\begin{theorem}\label{thm:schur-expansion}
For any fixed partition $\l$, there exists a bijection
\[\Phi:\SVRPP(\l)\longrightarrow\bigsqcup_\m \SSYT(\m)\times \OS^{\ell(\lambda)}(\l\rightarrow \m).\] 
Moreover, for $T\in \SVRPP(\l)$, 
if $\Phi(T)=(T', S)$, then \[\mathbf{x}^{\ircont(T)}\mathbf{w}^{\ex(T)}\mathbf{t}^{\ceq(T)}=\mathbf{x}^{\wt(T')} {\wt(S)}.\]
\end{theorem}

As a direct consequence of Theorem \ref{thm:schur-expansion}, we obtain a formula  for the coefficients $K_{\l,\m}(\mathbf{t},\mathbf{w})$ in terms of oscillating sequences.

\begin{corollary}\label{ciri-j0}
For partitions $\lambda$ and $\mu$, we have   
$$ 
K_{\l,\m}(\mathbf{t},\mathbf{w}) =  \sum_\alpha\sum_\beta f_{\l,\m}^{(\alpha,\beta)}\mathbf{w}^\alpha\mathbf{t}^\beta,$$
where \[f_{\l,\m}^{(\alpha,\beta)}=\#\{S\in \OS^{\ell(\l)}(\l\rightarrow \m):\ \wt(S)=\mathbf w^{\alpha} \mathbf t^{\beta}\}.\]

\end{corollary}

To prove  Theorem \ref{thm:schur-expansion}, we introduce an algorithm, called the \textit{dilation--contraction algorithm}, performing on   set-valued reverse plane partitions. The idea behind this  is mainly inspired by the operations separately on set-valued tableaux and reverse plane partitions   investigated  by 
 Bandlow and Morse \cite{Bandlow}.

We first  define the dilation operation. 
For a vector $v=(v_1,v_2,\ldots)$ and $r\geq 1$, we use $v_{\geq r}$ to denote the truncation $v_{\geq r}=(v_{r}, v_{r+1},\ldots)$ starting  from position $r$. 
For $r, s\geq 1$, let $\SVRPP(\lambda, r, s)$   denote the subset of $\SVRPP(\lambda)$ including  set-valued reverse plane partitions with $\ceq_{\ge r}(T)=\mathbf 0$ and $\ex_{\ge s}(T)=\mathbf 0$.
In other words, every box of $T$ in or below row $r$ has no redundant entry, and   every box of $T$ in or below row $s$ is filled with a single number.

For $k\geq 1$,  the \textit{{$k$-th dilation}} operation, denoted  $\mathrm{di}_k$, is performed on the set  $\SVRPP(\lambda, k, k+1)$. Let  $T\in\SVRPP(\lambda, k, k+1)$.  We use  $T_{\ge k+1}$ to denote the subtableau of $T$ occupying the boxes in and below  row $k+1$. 
Keep in mind that  $T_{\ge k+1}$ itself  is an ordinary semistandard tableau.  The image $\mathrm{di}_k(T)$ of  {$k$-th dilation}  acting on $T$ is defined as follows.
\begin{itemize}
    \item 
Consider the entries in the $k$-th  row   of $T$ by ignoring the smallest entry in each box. Suppose that such entries are $x_1>\cdots>x_m$. 
Remove    the entries $x_1,\ldots, x_m$    from largest to smallest, and meanwhile  insert these entries   successively  into $T_{\ge k+1}$ by applying the RSK algorithm. 
\end{itemize}
For example, in Figure \ref{fig-aghrn-2}, we apply the dilation algorithm  to a set-valued reverse partition of shape $\lambda=(2,2,1,1)$ and with $k=2$. In the process, we insert $(x_1,x_2,x_3)=(4,3,2)$  successively into $T_{\geq 3}$.
 
\begin{figure}[h t]
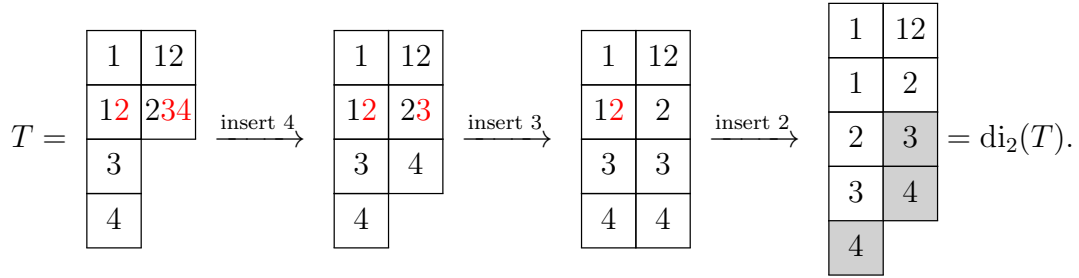

    \centering   
\ytableausetup{boxsize=1.7em}

\[
\begin{aligned}
T
&=
\ctab{
\begin{ytableau}
1&12\\
1\textcolor{red}{2} & 2\textcolor{red}{3}\textcolor{red}{4}  \\
3   \\
4
\end{ytableau}}
\;\xrightarrow{\text{insert $4$}}\;
\ctab{
\begin{ytableau}
1&12\\
1\textcolor{red}{2} & 2\textcolor{red}{3}  \\
3 & 4  \\
4
\end{ytableau}}
\;\xrightarrow{\text{insert $3$}}\;
\ctab{
\begin{ytableau}
1&12\\
1\textcolor{red}{2} & 2  \\
3 & 3 \\
4 & 4
\end{ytableau}}
\;\xrightarrow{\text{insert $2$}}\;
\ctab{
\begin{ytableau}
1&12\\
1 & 2 \\
2 & *(routegray)3 \\
3 & *(routegray)4 \\
*(routegray)4
\end{ytableau}}
=
\mathrm{di}_2(T).
\end{aligned}
\]
\caption{An illustration of the dilation algorithm.}
    \label{fig-aghrn-2}
\end{figure}

\begin{lemma}\label{di}
Fix a shape $\lambda$.  The map $\mathrm{di}_k$ is a bijection from $\SVRPP(\lambda, k, k+1)$ to the union of $\SVRPP(\mu, k, k)$, where $(\lambda, \mu)$ is a $k$-rising pair.

\end{lemma}

\begin{myproof}
Let $T\in \SVRPP(\lambda, k, k+1)$, and $T'=\mathrm{di}_k(T)$. We first check that $T'$ belongs to  $\SVRPP(\mu, k, k)$, where $(\lambda, \mu)$ is a $k$-rising pair. Assume that $x_1>\cdots>x_m$ are the entries row $k$, as used in the definition of the dilation. Let $P_1$ be obtained from $T$ by removing $x_1$ and then inserting $x_1$ into $T_{\geq k+1}$. It is straightforward to check that  $P_1$ is a valid set-valued reverse plane partition with $\ceq_{\ge k}(P_1)=\mathbf 0$. Doing the same for $x_2,\ldots, x_m$ yields the image  $T'$ which is also a valid set-valued reverse plane partition with $\ceq_{\ge k}(T')=\mathbf 0$. It is also clear that  $\ex_{\ge k}(T')=\mathbf 0$. 
In view of Proposition \ref{rsk}, the   difference of shapes of $T'$ and $T$ is a vertical strip. So we conclude that $T'\in \SVRPP(\mu, k, k)$, where $(\lambda, \mu)$ is a $k$-rising pair.

Conversely, let $T'$ $\in \SVRPP(\mu, k, k)$, where $(\lambda, \mu)$ is a $k$-rising pair. Do the reverse RSK to the  entries in the vertical strip $\mu/\lambda$ from bottom to top, giving rise to a sequence $x_1>\cdots>x_m$ of bumped entries.  Place these $x_1>\cdots>x_m$ into the boxes in  row $k$ in a unique way such that  row $k$ keeps weakly increasing. It is not hard to check that the final outcome is a set-valued reverse plane partition $T$ in   $\SVRPP(\lambda, k, k+1)$. 
Moreover, by the above construction, we have $\mathrm{di}_{k}(T)=T'$. This allows us to conclude  the lemma. 
\end{myproof}

We next  define the contraction operation. For $k\geq 1$,  the \textit{{$k$-th contraction}}, denoted  $\mathrm{co}_k$, acts on the set  $\SVRPP(\lambda, k, k)$. Let  $T\in \SVRPP(\lambda, k, k)$. Equivalently, the subtableau $T_{\geq k}$ is an ordinary semistandard tableau.   The image  $\mathrm{co}_k(T)$ is constructed as follows. 
\begin{itemize}
    \item Replace  each box in the $k$-th row of $T$, whose entry equals the largest  entry of the set  immediately above it in row $k-1$, by an empty box. Then apply the jdt  slides to move these empty boxes  down and to the right until they exit the diagram. The slide order is from right to left, thus every original empty box can be viewed as an inside corner.
    This procedure can be best understood by an example as given in Figure     \ref{fig-aghafe} where $k=2$. 
\end{itemize}
\begin{figure}[h t]
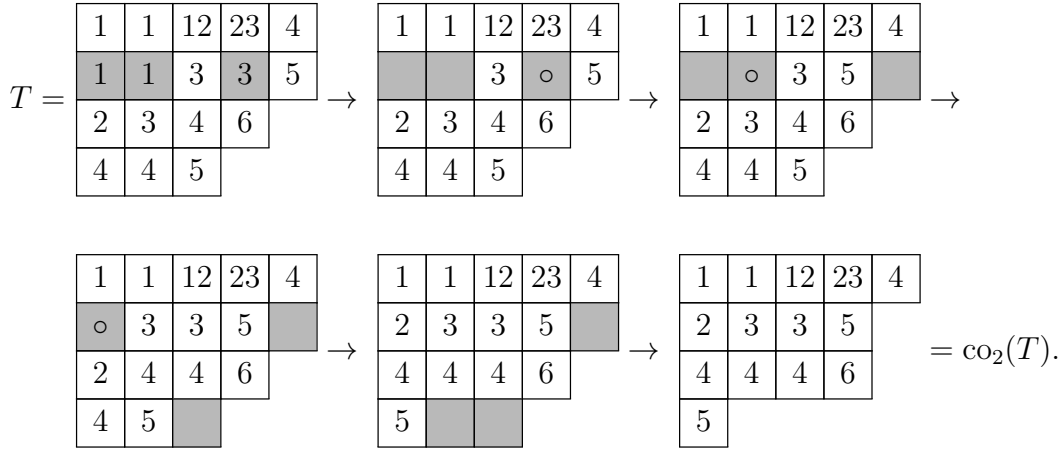

    \centering 
\ytableausetup{boxsize=1.5em}

\[
\begin{aligned}
T
=&
\ctab{
\begin{ytableau}
1 & 1 & 12 &23 &4\\
*(gray!55)1 & *(gray!55)1 & 3 &*(gray!55)3 &5 \\
2 & 3 & 4 &6\\
4 & 4 & 5 \\
\end{ytableau}}
\rightarrow
\ctab{
\begin{ytableau}
1 & 1 & 12 &23 &4\\
*(gray!55) & *(gray!55) & 3 &*(gray!55)\circ &5 \\
2 & 3 & 4 &6\\
4 & 4 & 5 \\
\end{ytableau}}
\rightarrow
\ctab{
\begin{ytableau}
1 & 1 & 12 &23 &4\\
*(gray!55) & *(gray!55)\circ & 3 &5 &*(gray!55) \\
2 & 3 & 4 &6\\
4 & 4 & 5 \\
\end{ytableau}}
\rightarrow\\[1.5em]
&
\ctab{
\begin{ytableau}
1 & 1 & 12 &23 &4\\
*(gray!55)\circ  & 3 & 3 &5 &*(gray!55) \\
2 & 4 & 4 &6\\
4 & 5 & *(gray!55)\\
\end{ytableau}}
\rightarrow
\ctab{
\begin{ytableau}
1 & 1 & 12 &23 &4\\
2  & 3 & 3 &5 &*(gray!55) \\
4 & 4 & 4 &6\\
5 & *(gray!55) & *(gray!55)\\
\end{ytableau}}
\rightarrow
\ctab{
\begin{ytableau}
1 & 1 & 12 &23 &4\\
2  & 3 & 3 &5  \\
4 & 4 & 4 &6\\
5 \\
\end{ytableau}}
=
\mathrm{co}_2(T).
\end{aligned}
\]
\caption{An illustration of the contraction  algorithm.}
    \label{fig-aghafe}
\end{figure}

\begin{lemma}\label{co}
Fix a shape $\lambda$.  The map $\mathrm{co}_k$ is a bijection from $\SVRPP(\lambda, k, k)$ to the union of $\SVRPP(\mu, k-1, k)$, where $(\lambda, \mu)$ is a $(k-1)$-falling pair.
\end{lemma}

\begin{myproof}
Let  $T\in \SVRPP(\l, k, k)$.    We can forget all the entries of $T$ in row $k-1$ except the largest one in every box and add them back after the whole process. 
Suppose that the empty boxes in row $k$ are $(k, j_1),\ldots, (k,j_m)$, listed from right  to left. By the construction of $\co_k$, we successively perform the jdt algorithm starting from the empty  boxes in the order  $(k, j_1),\ldots, (k,j_m)$. A crucial observation is that   the $m$ slide paths do not intersect,  distributed from upper right to lower left.  
Denote by $T'$ the resulting filling.  Now it is a routine task to check that $T'$ is a set-valued reverse plane partition belonging to $\SVRPP(\mu, k-1, k)$, where $\lambda/\mu$ is a horizontal strip (so $(\lambda, \mu)$ is a $(k-1)$-falling pair).

Conversely, consider a $T'\in \SVRPP(\mu, k-1, k)$, where $(\lambda, \mu)$ is a $(k-1)$-falling pair. Then do the reverse jdt slides to these empty boxes in the horizontal strip $\lambda/\mu$ from left to right, and terminate at row $k$ (which means the next jdt slide will move the empty box up into row $k-1$). Then fill every  empty box in row $k$ with the largest  entry in the box immediately  above it. It is not hard to check that the final outcome is a set-valued reverse plane partition $T$ in $\SVRPP(\lambda, k, k)$. Moreover, the above construction is indeed the reverse map of   $\co_{k}$. So the proof is complete. 
\end{myproof}

We can now describe our  dilation-contraction  algorithm 
which transforms a  set-valued reverse plane  partition of shape $\lambda$ to a semistandard tableau of shape $\mu$ together with an $\ell(\lambda)$-oscillating sequence from $\lambda$ to $\mu$. 
Let $T\in \SVRPP(\lambda)$. We first use the dilation and contraction operations to generate a sequence $(T_0, T_1,\ldots, T_{2\ell(\lambda)-1})$ of set-valued reverse plane partitions.  Set $l=\ell(\lambda)$ and $T_0=T$. 
The tableaux $T_i$ for $i=1,\ldots, 2l-1$ are constructed as follows.
\begin{itemize}
    \item 
 Let $1\leq i\leq 2l-1$. If $i=2k-1$ is odd, then let 
 \[
 T_i=\mathrm{di}_{l+1-k}(T_{i-1})
 \]
be obtained from $T_{i-1}$ by applying the $(l+1-k)$-th dilation.   If $i=2k$ is even, then let    
   \[
 T_i=\mathrm{co}_{l+1-k}(T_{i-1})
 \]
be obtained from $T_{i-1}$ by applying the $(l+1-k)$-th contraction.  
That is,
\[
T_0 \xrightarrow{\text{$\mathrm{di}_l$}}T_1\xrightarrow{\text{$\mathrm{co}_l$}}T_2\xrightarrow{\text{$\mathrm{di}_{l-1}$}}T_3\xrightarrow{\text{$\mathrm{co}_{l-1}$}}T_4\xrightarrow{\text{$\mathrm{di}_{l-2}$}}\cdots\xrightarrow{\text{$\mathrm{co}_{2}$}}T_{2l-2}\xrightarrow{\text{$\mathrm{di}_{1}$}}T_{2l-1}.
\]

\end{itemize}
Figure   \ref{fig-fb-9jnfj} gives  a concrete example to illustrate the dilation-contraction algorithm.
\begin{figure}[h t]
    \centering  
\[
\begin{aligned}
&T_0
=\ctab{
\begin{ytableau}
12 & 23 & 3 \\
2  & 3  & 45 \\
2\textcolor{red}{4}
\end{ytableau}}
\xrightarrow{\,\operatorname{di}_3\,}
T_1
=\ctab{
\begin{ytableau}
12 & 23 & 3 \\
2  & 3  & 45 \\
\textcolor{blue}{2} \\
4
\end{ytableau}}
\xrightarrow{\,\operatorname{co}_3\,}
T_2
=\ctab{
\begin{ytableau}
12 & 23 & 3 \\
2  & 3  & 4\textcolor{red}{5} \\
4
\end{ytableau}}
\xrightarrow{\,\operatorname{di}_2\,}\\
&T_3
=\ctab{
\begin{ytableau}
12 & 23 & 3 \\
\textcolor{blue}{2}  & \textcolor{blue}{3}  & 4 \\
4  & 5
\end{ytableau}}
\xrightarrow{\,\operatorname{co}_2\,}
T_4
=\ctab{
\begin{ytableau}
1\textcolor{red}{2} & 2\textcolor{red}{3} & 3 \\
4  & 4 \\
5
\end{ytableau}}
\xrightarrow{\,\operatorname{di}_1\,}
T_5
=\ctab{
\begin{ytableau}
1 & 2 & 3 \\
2 & 4 \\
3 \\
4 \\
5
\end{ytableau}}.
\end{aligned}
\]
\caption{An illustration of the dilation-contraction  algorithm.}
    \label{fig-fb-9jnfj}
\end{figure}

We can now construct the map $\Phi$
in Theorem \ref{thm:schur-expansion}. For $T\in \SVRPP(\lambda)$, let  $(T_0, T_1,\ldots, T_{2\ell(\lambda)-1})$ be the sequence of set-valued reverse partitions after applying the dilation-contraction algorithm.
Notice that $T_{2\ell(\lambda)-1}$ is an ordinary semistandard tableau. For $0\leq i\leq 2\ell(\lambda)-1$, write $\lambda^i$ for the shape of $T_i$. Define $\Phi(T)=(T', S)$ where $T'=T_{2\ell(\lambda)-1}$ and 
\[
S=(\lambda^0, \lambda^1,\ldots, \lambda^{2\ell(\lambda)-1}).
\]

\noindent
{\it Proof of Theorem \ref{thm:schur-expansion}.}
By Lemmas  \ref{di} and \ref{co}, it is an easy task to justify   that the map $\Phi$ is indeed a bijection from  $\SVRPP(\l)$ to 
\[\bigsqcup_\m \SSYT(\m)\times \OS^{\ell(\lambda)}(\l\rightarrow \m).\] 
Moreover, from the definitions of $\ircont(T)$, $\ex(T)$, $\ceq(T)$ and $\wt(S)$, we can directly check    that 
\[\mathbf{x}^{\ircont(T)}\mathbf{w}^{\ex(T)}\mathbf{t}^{\ceq(T)}=\mathbf{x}^{\wt(T')} {\wt(S)}.\]
This completes the proof. 
\qed

\begin{remark}\label{ren-jai}
For a  partition $\lambda$, let  $\overline{\l}=(\overline{\l}_1,\overline{\l}_2,\ldots)$ be given by
$$\overline{\l}_i:=\min\{\l_j + j - 1: 1\leq j\leq i\}.$$ 
By  \cite[Section 4]{GuoKangLiu2025}, 
\[
(\l_1)\subseteq \m\subseteq \overline{\l}\Longleftrightarrow K_{\l,\m}(\mathbf{t},\mathbf{w})\neq 0.
\]
(In \cite[Section 4]{GuoKangLiu2025}, we defined $\overline{\l}^{(n)}$ and here we need to let $n\rightarrow \infty$ to obtain $\overline{\l}$.)
It should be noted that $\overline{\l}$ is   regarded as a weakly decreasing sequence since it  has  infinitely many nonzero parts, and the notation $\m\subseteq \overline{\l}$ means that $\mu_i\leq \overline{\l}_i$ for $i\geq 1$.
Combining Corollary \ref{ciri-j0},
we have 
\[(\l_1)\subseteq \m\subseteq \overline{\l} \Longleftrightarrow \OS^{\ell(\lambda)}(\l\rightarrow \m)\neq \emptyset  .\]

\end{remark}

\subsection{Distributive lattice of oscillating sequences}

We show that oscillating sequences admit a natural distributive lattice structure.  A crucial  observation is as follows.

\begin{lemma}\label{lemn-nb-0}
If $(\lambda, \rho)$ and $(\mu, \gamma)$ are   $k$-rising pairs (resp., $k$-falling pairs), then $(\lambda \vee \mu, \rho \vee \gamma)$ and 
$(\lambda \wedge \mu, \rho \wedge \gamma)$ 
are also  $k$-rising pairs  (resp., $k$-falling pairs).
\end{lemma}

\begin{myproof}
We only  verify the case for $k$-rising pairs. The  proof for  $k$-falling pairs can be dealt with in an analogous  manner. 

For
$1 \le i \le k$, it follows from the definition that     $\lambda_i = \rho_i$ and $\mu_i = \gamma_i$, and so we have 
\[
(\lambda \vee \mu)_i
=
\max\{\lambda_i,\mu_i\}
=
\max\{\rho_i,\gamma_i\}
=
(\rho \vee \gamma)_i,
\]
and
\[
(\lambda \wedge \mu)_i
=
\min\{\lambda_i,\mu_i\}
=
\min\{\rho_i,\gamma_i\}
=
(\rho \wedge \gamma)_i.
\]
It remains to show that both $(\rho \vee \gamma)/(\lambda \vee \mu)$ and $(\rho \wedge \gamma)/(\lambda \wedge \mu)$ 
are  vertical strips.

Notice that if $\lambda \subseteq \rho$ and
$\mu \subseteq \gamma$, then
\[
\lambda \vee \mu \subseteq \rho \vee \gamma
\qquad\text{and}\qquad
\lambda \wedge \mu \subseteq \rho \wedge \gamma.
\]
Since both $\rho/\lambda$ and $\gamma/\mu$ are   vertical strips, we see that for $j\geq 1$,
\[
\rho_j-\lambda_j\le 1,
\qquad
\gamma_j-\mu_j\le 1.
\] 
Thus, for  $j\geq 1$, we deduce that 
\begin{align*}
(\rho \vee \gamma)_j-(\lambda \vee \mu)_j
&=
\max\{\rho_j,\gamma_j\}
-
\max\{\lambda_j,\mu_j\}\\
&=
\max\{
\rho_j-\max\{\lambda_j,\mu_j\},
\,
\gamma_j-\max\{\lambda_j,\mu_j\}
\}\\
&\le
\max\{\rho_j-\lambda_j,\gamma_j-\mu_j\}
\le 1,
\end{align*}
and 
\begin{align*}
(\rho \wedge \gamma)_j-(\lambda \wedge \mu)_j
&=
\min\{\rho_j,\gamma_j\}
-
\min\{\lambda_j,\mu_j\}\\
&=
\min\{\rho_j,\gamma_j\}
+
\max\{-\lambda_j,-\mu_j\}\\
&=
\max\{
\min\{\rho_j,\gamma_j\}-\lambda_j,\,
\min\{\rho_j,\gamma_j\}-\mu_j
\}\\
&\le
\max\{\rho_j-\lambda_j,\gamma_j-\mu_j\}
\le 1,
\end{align*}
which implies that   $(\rho \vee \gamma)/(\lambda \vee \mu)$ and $(\rho \wedge \gamma)/(\lambda \wedge \mu)$ 
are  vertical strips. 
\end{myproof}

Denote 
\[
\OS^{l}(\l)=\bigcup_{\mu} \OS^l(\lambda\to\mu)=\bigcup_{\mu: (\l_1)\subseteq \m\subseteq \overline{\l}} \OS^l(\lambda\to\mu),
\]
where the second equality follows from 
Remark \ref{ren-jai}.
We also use  the notation
\[
\mathrm{OS}^l_k=\bigcup_{\lambda:\,\ell(\lambda)\le l,\,\lambda_1\le k}\mathrm{OS}^l(\lambda),
\]
and 
\begin{equation}\label{aghkhn-afn-0}
\OS^{l,m}_k=\bigcup_{\substack{\l:\ell(\l)\leq l\\\l_1\leq k}}\ \bigcup_{\substack{\mu:\ell(\m)\leq m\\ (\l_1)\subseteq \m\subseteq \overline{\l}}}\OS^l(\l\rightarrow \m).
\end{equation}

The following is a consequence of Lemma \ref{lemn-nb-0}.

\begin{corollary}\label{cln-sl}
For two $l$-oscillating sequences $S=(\lambda^0, \lambda^1,\ldots, \lambda^{2l-1})$ and $S'=(\mu^0, \mu^1,\ldots, \mu^{2l-1})$,
define 
\[
S\vee S'=(\lambda^0\vee \mu^0, \lambda^1\vee \mu^1,\ldots, \lambda^{2l-1}\vee \mu^{2l-1})
\]
and 
\[
S\wedge S'=(\lambda^0\wedge \mu^0, \lambda^1\wedge \mu^1,\ldots, \lambda^{2l-1}\wedge \mu^{2l-1}).
\]
Then $S\vee S'$ and $S\wedge S'$ are also $l$-oscillating sequences. With the above join and meet, the sets $\OS^l(\lambda\to\mu)$, $\OS^l(\lambda)$, 
$\OS^l_k$ and  $\OS^{l,m}_k$ are all distributive lattices. In particular, the first and the last are finite distributive lattices. 
\end{corollary}

\section{Multivariate Reuter--Lov\'asz--Saks inequality}\label{sec:lpp}

In this section, we   present  a {multivariate  Reuter--Lov\'asz--Saks type inequality} which is another  ingredient in the proof of   Theorem \ref{main-1}. 
For the history on the classical  Reuter--Lov\'asz--Saks inequality, we refer the reader to \cite{CCPS} for an overview.

Throughout this section, let $({L}, \wedge, \vee)$ be a {\it finite} distributive lattice. 
Fix variables $\mathbf{q}=(q_1, \ldots, q_N)$, and \emph{modular functions} 
$\mathbf{r}=(r_1, \ldots, r_N)$, that is, for $1\leq i\leq N$,  $r_i: {L} \to \mathbb{Z}_{\geq 0}$ satisfies 
\begin{equation}\label{req-12}
r_i(x) + r_i(y) = r_i(x \vee y) + r_i(x \wedge y) \quad \text{for } x, y \in {L}.
\end{equation}
For $x \in  {L}$, write
 $\mathbf q^{\mathbf{r}(x)} := q_1^{r_1(x)} \cdots q_N^{r_N(x)}$.
 For a function $\rho: {L} \to \mathbb{R}_{\geq 0}$ and subset $X \subseteq {L}$, define
\begin{equation*}
\rho_{\langle \mathbf q, \mathbf{r} \rangle}(X) := \sum_{x \in X} \rho(x) \mathbf q^{\mathbf{r}(x)} \in \mathbb{R}_{\geq 0}[q_1, \ldots, q_N].
\end{equation*}
For two polynomials $f, g$ in $\mathbf q$,   write $f\leq^{\mathbf q}g$
if
\[
g-f\in \mathbb R_{\ge 0}[\mathbf q].
\]

For two subset $X, Y\subseteq L$, let $X \vee Y=\{x\vee y\colon x\in X, y\in Y\}$ and 
$X \wedge  Y=\{x\wedge y\colon x\in X, y\in Y\}$. The following is a  multivariate Reuter--Lov\'asz--Saks type inequality. 

\begin{theorem}[multivariate RLS type  inequality]\label{WRLS}
Let $(L,\wedge,\vee)$ be a finite distributive lattice, and let
$r_1,\ldots,r_N:L\to \mathbb Z_{\geq 0}$
be modular functions satisfying \eqref{req-12}. 
Suppose that for every subset
$H\subseteq[n]:=\{1,\ldots, n\}$ , we are given subsets
\[
A_H,\ B_H,\ C_H,\ D_H\subseteq L
\]
such that for any  $I,J\subseteq[n]$, 
$A_I\vee B_J\subseteq C_{I\cup J}$ and $ 
A_I\wedge B_J\subseteq D_{I\cap J}.$
For a subset $X\subseteq L$, define
\[
X(\mathbf q)=\sum_{x\in X}\mathbf q^{\mathbf r(x)}.
\]
Then
\[
\sum_{H\subseteq[n]}
A_H(\mathbf q)B_{\overline{H}}(\mathbf q)
\le^{\mathbf q}
\sum_{H\subseteq[n]}
C_H(\mathbf q)D_{\overline{H}}(\mathbf q), \ \ \ \text{where $\overline{H}=[n]\setminus H$}.
\]
\end{theorem}

To give a proof of Theorem  \ref{WRLS}, we need the multivariate   Ahlswede--Daykin inequality due to Chan and Pak \cite[Theorem 6.1]{CP23}.

\begin{theorem}[\normalfont{\cite[Theorem 6.1]{CP23}}]\label{MAD} 
Let $({L}, \wedge, \vee)$ be a finite distributive lattice, and let 
$\alpha, \beta, \rho, \delta: {L} \to \mathbb{R}_{\geq 0}$ be nonnegative functions on ${L}$. 
Suppose that  
\begin{equation*}
\alpha(x) \cdot \beta(y) \leq \rho(x \vee y) \cdot \delta(x \wedge y) \quad \text{for } x, y \in  {L}.
\end{equation*}
Then, for  $X, Y \subseteq {L}$,  we have 
\begin{equation*}
\alpha_{\langle \mathbf q, \mathbf{r} \rangle}(X) \cdot \beta_{\langle \mathbf q, \mathbf{r} \rangle}(Y) \leq^ {\mathbf q}  \rho_{\langle \mathbf q, \mathbf{r} \rangle}(X \vee Y) \cdot \delta_{\langle \mathbf q, \mathbf{r} \rangle}(X \wedge Y) .
\end{equation*}
\end{theorem}

We  now  give a proof of Theorem  \ref{WRLS}.

\noindent
{\it Proof of Theorem  \ref{WRLS}}
We first  encode a subset   $H\subseteq [n]$ by extra variables 
 $\mathbf z=(z_1,\ldots,z_n)$, that is, write
\[
\mathbf z^H=\prod_{i\in H}z_i.
\]
Consider the product lattice
\[
\widetilde L=L\times 2^{[n]},
\]
with join and meet given by
\[
(x,H)\vee(y,J)=(x\vee y,H\cup J),
\ \ \ 
(x,H)\wedge(y,J)=(x\wedge y,H\cap J).
\]
With the above operations, 
$\widetilde L$ is a finite  distributive lattice.

We extend the modular functions on $L$ to $\widetilde L$ by setting
\[
\widetilde r_i(x,H)=r_i(x),
\qquad 1\le i\le N.
\]
Moreover, we introduce Boolean modular functions: for $1\le j\le n$, 
\[
\varepsilon_j(x,H)=
\begin{cases}
1,& j\in H,\\
0,& j\notin H.
\end{cases}
\]
Indeed, one has 
\[
\varepsilon_j(x,H)+\varepsilon_j(y,J)
=
\varepsilon_j(x\vee y,H\cup J)
+
\varepsilon_j(x\wedge y,H\cap J),
\]
because for each fixed $j\in [n]$,
\[
\mathbf 1_{j\in H}+\mathbf 1_{j\in J}
=
\mathbf 1_{j\in H\cup J}+\mathbf 1_{j\in H\cap J}.
\]
Now we see that  $\widetilde r_1,\ldots,\widetilde r_N,\varepsilon_1,\ldots,\varepsilon_n$
are modular functions on $\widetilde L$.

Define four subsets of $\widetilde L$ by
\begin{align*}
\mathcal A&=\{(a,H):H\subseteq[n],\ a\in A_H\},  \\[5pt]
\mathcal B&=\{(b,J):J\subseteq[n],\ b\in B_J\},\\[5pt]
\mathcal C&=\{(c,K):K\subseteq[n],\ c\in C_K\},\\[5pt]
\mathcal D&=\{(d,M):M\subseteq[n],\ d\in D_M\}.
\end{align*}
We claim that
\[
\mathcal A\vee\mathcal B\subseteq \mathcal C,
\qquad
\mathcal A\wedge\mathcal B\subseteq \mathcal D.
\]
This can be seen as follows. If $(a,H)\in\mathcal A$ and $(b,J)\in\mathcal B$, then
$a\in A_H$ and $b\in B_J$. By the assumption in the theorem, we have 
\[
a\vee b\in C_{H\cup J},
\qquad
a\wedge b\in D_{H\cap J}.
\]
Therefore
\[
(a,H)\vee(b,J)=(a\vee b,H\cup J)\in\mathcal C,
\]
and
\[
(a,H)\wedge(b,J)=(a\wedge b,H\cap J)\in\mathcal D.
\]
This verifies the claim. 

Applying Theorem  \ref{MAD}, along with the above claim,  to   $\widetilde L$ with  
$\alpha=\beta=\rho=\delta=1$ 
and with the modular functions
$\widetilde r_1,\ldots,\widetilde r_N,
\varepsilon_1,\ldots,\varepsilon_n$, we deduce  that 
\begin{align*}
\left(
\sum_{(a,H)\in\mathcal A}
\mathbf q^{{\mathbf  r}(a)}\mathbf z^H
\right)
\left(
\sum_{(b,J)\in\mathcal B}
\mathbf q^{{\mathbf  r}(b)}\mathbf z^J
\right)
&\le^{\mathbf q,\mathbf z}
\left(
\sum_{(c,K)\in\mathcal A\vee\mathcal B}
\mathbf q^{{\mathbf  r}(c)}\mathbf z^K
\right)
\left(
\sum_{(d,M)\in\mathcal \mathcal A\wedge\mathcal B}
\mathbf q^{{\mathbf  r}(d)}\mathbf z^M
\right)\\[5pt]
&\le^{\mathbf q,\mathbf z}
\left(
\sum_{(c,K)\in\mathcal C}
\mathbf q^{{\mathbf  r}(c)}\mathbf z^K
\right)
\left(
\sum_{(d,M)\in\mathcal D}
\mathbf q^{{\mathbf  r}(d)}\mathbf z^M
\right).
\end{align*}
Equivalently,
\[
\left(
\sum_{H\subseteq[n]}A_H(\mathbf q)\mathbf z^H
\right)
\left(
\sum_{J\subseteq[n]}B_J(\mathbf q)\mathbf z^J
\right)
\le^{\mathbf q,\mathbf z}
\left(
\sum_{K\subseteq[n]}C_K(\mathbf q)\mathbf z^K
\right)
\left(
\sum_{M\subseteq[n]}D_M(\mathbf q)\mathbf z^M
\right).
\]
We now take the coefficients of
$\mathbf z^{[n]}=z_1z_2\cdots z_n$ on both sides.
 Since each $\mathbf z^H$ is squarefree, the coefficient of
$\mathbf z^{[n]}$ on the left-hand side  is obtained precisely from pairs
$(H,J)$ satisfying
$H\cap J=\emptyset, 
H\cup J=[n]$.
Equivalently, $J=\overline{H}.$
Hence
\[
\left[\mathbf z^{[n]}\right]
\left(
\sum_{H\subseteq[n]}A_H(\mathbf q)\mathbf z^H
\right)
\left(
\sum_{J\subseteq[n]}B_J(\mathbf q)\mathbf z^J
\right)
=
\sum_{H\subseteq[n]}A_H(\mathbf q)B_{\overline{H}}(\mathbf q).
\]
Similarly,
\[
\left[\mathbf z^{[n]}\right]
\left(
\sum_{K\subseteq[n]}C_K(\mathbf q)\mathbf z^K
\right)
\left(
\sum_{M\subseteq[n]}D_M(\mathbf q)\mathbf z^M
\right)
=
\sum_{K\subseteq[n]}C_K(\mathbf q)D_{\overline{K}}(\mathbf q).
\]
Note that  the inequality before taking coefficients is coefficientwise in
both $\mathbf q$ and $\mathbf z$. Hence  taking the coefficient of
$\mathbf z^{[n]}$ preserves coefficientwise nonnegativity in
$\mathbf q$. This implies that 
\[
\sum_{H\subseteq[n]}
A_H(\mathbf q)B_{\overline{H}}(\mathbf q)
\le^{\mathbf q}
\sum_{H\subseteq[n]}
C_H(\mathbf q)D_{\overline{H}}(\mathbf q),
\]
and so  the proof is complete. 
\qed

\section{Schur orchestra inequality}\label{vari-srche}

In this section, we recall the {Schur orchestra inequality} developed  by Chan, Chen, Pak and Soskin \cite{CCPS}. We then  present a variation of this inequality which will be needed in the proof of Theorem \ref{main-1}.

 Let
\[
\mathcal E=\{1,\ldots,\ell\}\times\{0\},
\qquad
\mathcal F=\{1,\ldots,\ell\}\times\{1\}
\]
be two disjoint copies of $[\ell]$. Denote by $\mathcal E\sqcup\mathcal F$
 their disjoint union.
Let
\[
O: 2^{\mathcal E\sqcup\mathcal F}\longrightarrow \mathbb R
\]
be a real-valued function on the Boolean lattice of subsets of
$\mathcal E\sqcup\mathcal F$.
We say  that \(O\) satisfies the orchestra inequalities if for every set partition $\mathcal E\sqcup\mathcal F=B_1\sqcup\cdots\sqcup B_t$, 
\[
\sum_{H\subseteq[t]}
O\!\left(\bigsqcup_{h\in H}B_h\right)\ge0.
\]

Suppose that 
\[
\alpha,\beta,\gamma,\delta\in\mathbb Z_{\geq 0}^\ell.
\]
For a subset \(E\subseteq\mathcal E\), define
$\lambda(E)=(\lambda_1(E),\ldots,\lambda_\ell(E))$
by
\begin{equation}\label{eq-algm-1}
\lambda_i(E)=
\begin{cases}
\alpha_i+\gamma_i, & i\in E,\\
\alpha_i, & i\notin E,
\end{cases}
\end{equation}
where we simply use $i\in E$ to represent  $(i,0)\in E$. 
Similarly, for a subset \(F\subseteq\mathcal F\), define
 $\mu(F)=(\mu_1(F),\ldots,\mu_\ell(F))$
 by
\[
\mu_j(F)=
\begin{cases}
\beta_j+\delta_j, & j\in F,\\
\beta_j, & j\notin F.
\end{cases}
\]

For \(E\subseteq\mathcal E\) and \(F\subseteq\mathcal F\), write
$\overline E=\mathcal E\setminus E$ and  
$\overline F=\mathcal F\setminus F$.  We also adopt  the following convention: if \(\lambda\) or \(\mu\) is not a partition,
or if \(\mu\not\subseteq\lambda\), then
 $s_{\lambda/\mu}=0$.

\begin{theorem}[\text{Schur orchestra inequality \cite[Theorem 5.2]{CCPS}}]\label{cik-oo}
Let
\[
O:2^{\mathcal E\sqcup\mathcal F}\longrightarrow\mathbb R
\]
be a function satisfying the  orchestra inequalities. Then
\[
\sum_{\substack{E\subseteq\mathcal E\\ F\subseteq\mathcal F}}
O(E\sqcup F)\,
s_{\lambda(E)/\mu(F)}(\mathbf x)\,
s_{\lambda(\overline E)/\mu(\overline F)}(\mathbf x)
\geq_s0.
\]
\end{theorem}

What we require is a variation  of the Schur orchestra inequality.

\begin{theorem}
\label{cor:straight-complementary-orchestra}
Let \(I=[\ell]\), and let
$O:2^I\longrightarrow \mathbb R$
be a function satisfying the orchestra inequalities: for every set partition
$I=B_1\sqcup\cdots\sqcup B_t$,
\[
\sum_{H\subseteq[t]}
O\left(\bigcup_{h\in H}B_h\right)\ge 0.
\]
Let \(\alpha,\gamma\in\mathbb Z_{\ge0}^{\ell}\). For \(E\subseteq I\), define $\lambda(E)$ as in \eqref{eq-algm-1}. Then
\[
\sum_{E\subseteq I}
O(E)\,
s_{\lambda(E)}(\mathbf x)\,
s_{\lambda(I\setminus E)}(\mathbf x)
\ge_s 0.
\]
\end{theorem}

\begin{myproof}
Let
\[
\mathcal E=I\times\{0\},
\qquad
\mathcal F=I\times\{1\}
\]
be two disjoint copies of \(I\). For \(A\subseteq\mathcal E\), define the projection 
\[
\pi(A)=\{i\in I:(i,0)\in A\}.
\]
Construct  a function
\[
\widetilde O:2^{\mathcal E\sqcup\mathcal F}\longrightarrow\mathbb R
\]
by
\[
\widetilde O(A\sqcup B)
=
2^{-\# I}\,O(\pi(A)),
\qquad
A\subseteq\mathcal E,\ B\subseteq\mathcal F.
\]
Thus \(\widetilde O\) depends only on the part   lying in the first copy
\(\mathcal E\).

We first prove that \(\widetilde O\) satisfies the orchestra inequalities on
\(\mathcal E\sqcup\mathcal F\). Let
\[
\mathcal E\sqcup\mathcal F
=
C_1\sqcup\cdots\sqcup C_t
\]
be an arbitrary set partition. For each \(j\in[t]\), define
\[
D_j=\pi(C_j\cap\mathcal E)\subseteq I.
\]
The nonempty sets among \(D_1,\ldots,D_t\) are pairwise disjoint and have union
\(I\). Let
\[
M=\{j\in[t]:D_j\neq\emptyset\},
\qquad m=\# M.
\]
Then
\[
I=\bigsqcup_{j\in M}D_j.
\]
For every \(H\subseteq[t]\), we have
\[
\pi\left(
\left(\bigsqcup_{h\in H}C_h\right)\cap\mathcal E
\right)
=
\bigcup_{h\in H}D_h
=
\bigcup_{h\in H\cap M}D_h.
\]
Hence
\begin{align}\label{akl-8-kbn}
\sum_{H\subseteq[t]}
\widetilde O\left(\bigsqcup_{h\in H}C_h\right)
&=
2^{-\# I}
\sum_{H\subseteq[t]}
O\left(\bigcup_{h\in H}D_h\right)\nonumber\\[5pt]
&=
2^{-\# I}\,2^{t-m}
\sum_{H'\subseteq M}
O\left(\bigcup_{h\in H'}D_h\right).
\end{align}
Indeed, after fixing \(H'=H\cap M\), the remaining indices in
\([t]\setminus M\) can be chosen arbitrarily, and they do not change the union
\(\bigcup_{h\in H}D_h\), since \(D_j=\emptyset\) for \(j\notin M\).

Since the nonempty \(D_j\)'s form a set partition of \(I\), the sum in \eqref{akl-8-kbn} is
nonnegative by the orchestra inequalities satisfied by  \(O\). Therefore \(\widetilde O\)
satisfies the orchestra inequalities on \(\mathcal E\sqcup\mathcal F\).

We now apply Theorem \ref{cik-oo} to \(\widetilde O\) where we set $\beta=\delta=\mathbf 0$.
For  \(A\subseteq\mathcal E\) and  \(B\subseteq\mathcal F\), we see that  
\[
s_{\lambda(A)/\mu(B)}(\mathbf x)
=
s_{\lambda(\pi(A))}(\mathbf x),
\]
and 
\[
s_{\lambda(\mathcal E\setminus A)/\mu(\mathcal F\setminus B)}(\mathbf x)
=
s_{\lambda(I\setminus\pi(A))}(\mathbf x).
\]
 Theorem \ref{cik-oo} therefore gives
\begin{align}\label{hak-0kj-0}
0
&\le_s
\sum_{A\subseteq\mathcal E}
\sum_{B\subseteq\mathcal F}
\widetilde O(A\sqcup B)\,
s_{\lambda(\pi(A))}(\mathbf x)\,
s_{\lambda(I\setminus\pi(A))}(\mathbf x)\nonumber\\[5pt]
&=
\sum_{A\subseteq\mathcal E}
\sum_{B\subseteq\mathcal F}
2^{-\# I}O(\pi(A))\,
s_{\lambda(\pi(A))}(\mathbf x)\,
s_{\lambda(I\setminus\pi(A))}(\mathbf x).
\end{align}
For a fixed \(A\subseteq\mathcal E\), the summand is independent of
\(B\subseteq\mathcal F\), and there are \(2^{\# I}\) choices of \(B\). Hence the
normalizing factor \(2^{-\# I}\) cancels the sum over \(B\). Since
\(A\mapsto \pi(A)\) is a bijection from subsets of \(\mathcal E\) to subsets of
\(I\),  \eqref{hak-0kj-0}  becomes
\[
\sum_{E\subseteq I}
O(E)\,
s_{\lambda(E)}(\mathbf x)\,
s_{\lambda(I\setminus E)}(\mathbf x),
\]
as desired. 
\end{myproof}

\section{Proof of Theorem \ref{main-1}}\label{sec:schur-positive}

We will finish the proof of Theorem \ref{main-1} in this section. The main ingredients include  Corollaries  \ref{ciri-j0} and  \ref{cln-sl}, Theorem \ref{WRLS}, as well as  Theorem \ref{cor:straight-complementary-orchestra}.

\begin{theorem}[=Theorem \ref{main-1}]\label{thm:hgp-lpp}
For    partitions $\lambda,\mu$, we have 
\[
H_\lambda(\mathbf x;\mathbf t,\mathbf w)\,
H_\mu(\mathbf x;\mathbf t,\mathbf w)
\leq _{s}^{\mathbf t,\mathbf w}
H_{\lambda\vee\mu}(\mathbf x;\mathbf t,\mathbf w)\,
H_{\lambda\wedge\mu}(\mathbf x;\mathbf t,\mathbf w).
\]
\end{theorem}

\begin{myproof}
Let $l=\max\{\ell(\l),\ell(\mu)\}$. 
For an $l$-oscillating sequence 
\[
S=(\lambda^0,\lambda^1,\ldots,\lambda^{2l-1}),
\]
define $\mathbf r(S)=(r_1(S),\ldots,r_{2l-1}(S))$ by setting 
\[
r_{2i-1}(S)=|\lambda^{2i-1}/\lambda^{2i-2}|\ \text{for $1\le i\le l$},\ \ 
r_{2j}(S)=|\lambda^{2j-1}/\lambda^{2j}| \ \text{ for $1\le j\le l-1$}.
\]
For two $l$-oscillating sequences $S$  and $S'$, it is readily checked that for $1\leq i\leq 2l-1$,
\begin{equation}\label{abkb-8-i}
  r_i(S)+ r_i(S')=  r_i(S\vee S')+  r_i(S\wedge S').   
\end{equation}

Write   $\mathbf q=(q_1,\ldots,q_{2l-1})$. Then 
\[
\mathbf q^{\mathbf r(S)}
=\prod_{1\le i\le l}q_{2i-1}^{r_{2i-1}(S)}\prod_{1\le j\le l-1}q_{2j}^{r_{2j}(S)}.
\]
For partitions $\lambda,\alpha$, define
\[
K_{\lambda,\alpha}(\mathbf q)
=
\sum_{S\in \OS^l(\lambda\to\alpha)}
\mathbf q^{\mathbf r(S)}.
\]
For $\mathbf a\in\mathbb Z_{\geq 0}^{2l-1}$, set 
\[
F_\alpha^{\mathbf a}(\beta)
=
\#\{S\in \OS^l(\alpha\to\beta): \mathbf r(S)=\mathbf a\}.
\]
Thus
\[
K_{\alpha,\beta}(\mathbf q)
=
\sum_{\mathbf a}
F_\alpha^{\mathbf a}(\beta)\mathbf q^{\mathbf a}.
\]

Let 
\[
{H}_{\lambda}(\mathbf x;\mathbf q)
=
\sum_\alpha K_{\lambda,\alpha}(\mathbf q)s_{\alpha}(\mathbf x),
\]
which, together with Corollary \ref{ciri-j0}, specializes  to  ${H}_{\lambda}(\mathbf x;\mathbf t, \mathbf w)$ after the substitutions
\begin{equation}\label{agle-789}
q_{2i-1}=w_{l-i+1}\ (1\le i\leq l),
\qquad
q_{2j}=t_{l-j}\ (1\le j\le l-1).
\end{equation}

\noindent
{\it Remark.} Strictly speaking, in  Corollary \ref{ciri-j0}, we require  $l=\ell(\lambda)$. In the case when $l>\ell(\lambda)$, we notice that  the first $2(l-\ell(\lambda))$ shapes in any oscillating sequence in $\OS^l(\lambda\to\alpha)$ are equal, and so there is a  bijection from  $\OS^l(\lambda\to\alpha)$ to  $\OS^{\ell(\lambda)}(\lambda\to\alpha)$ by ignoring the initial  $2(l-\ell(\lambda))$ shapes. In particular, after the substitutions in \eqref{agle-789}, ${H}_{\lambda}(\mathbf x;\mathbf q)$ still specializes to ${H}_{\lambda}(\mathbf x;\mathbf t, \mathbf w)$.

For $\zeta \in\mathbb Z_{\geq 0}^{2l-1}$, denote 
\[
L_\zeta(\nu,\gamma)
=
\sum_{\mathbf a+\mathbf b=\zeta}
F_\lambda^{\mathbf a}(\nu)
F_\mu^{\mathbf b}(\gamma),
\]
and
\[
R_\zeta(\nu,\gamma)
=
\sum_{\mathbf a+\mathbf b=\zeta}
F_{\lambda\vee\mu}^{\mathbf a}(\nu)
F_{\lambda\wedge\mu}^{\mathbf b}(\gamma).
\]
Then the coefficient of $\mathbf q^\zeta$ in
\[
H_{\lambda\vee\mu}(\mathbf x;\mathbf q) H_{\lambda\wedge\mu}(\mathbf x;\mathbf q)
-
H_\lambda(\mathbf x;\mathbf q) H_\mu(\mathbf x;\mathbf q)
\]
is
\begin{equation}\label{xx-k0}
\Delta_\zeta
=
\sum_{\nu,\gamma}
\bigl(R_\zeta(\nu,\gamma)-L_\zeta(\nu,\gamma)\bigr)
s_\nu s_\gamma.
\end{equation}
Therefore, it is  enough to prove that for every fixed $\zeta$, $\Delta_\zeta$ is Schur positive.  The proof will be divided into three steps.

Step 1.  
We first explain  that the  sum in \eqref{xx-k0} is finite for each fixed
  $\zeta\in \mathbb Z_{\geq 0}^{2l-1}$. It follows from definition that if 
$F^{\mathbf a}_\lambda(\nu)\neq 0$, then there exists
\[
S=(\lambda^0,\lambda^1,\ldots,\lambda^{2l-1})
\in \OS^l(\lambda\to\nu)
\]
with $\mathbf r(S)=\mathbf a$. Recall that 
\[
a_{2i-1}=|\lambda^{2i-1}/\lambda^{2i-2}|\ (1\le i\leq l),
\qquad
a_{2j}=|\lambda^{2j-1}/\lambda^{2j}|\ (1\le j\leq l-1).
\]
Hence
\[
|\nu|
=
|\lambda|
+
\sum_{i=1}^l a_{2i-1}
-
\sum_{j=1}^{l-1} a_{2j}.
\]
Thus, for fixed $\lambda$ and $\mathbf a$, the endpoint shape $\nu$ must have
a fixed size, and so there  are  finitely many such partitions $\nu$.

For fixed $\zeta$, there are only finitely many decompositions
$\mathbf a+\mathbf b=\zeta$. It follows that
$L_\zeta(\nu,\gamma)$ and $R_\zeta(\nu,\gamma)$ are nonzero for only
finitely many pairs $(\nu,\gamma)$. 

The following two steps will  borrow  some ideas in the proofs of \cite[Theorem 1.7]{CCPS} and \cite[Theorem 8.1]{CCPS}.

Step 2.   We now group the summands of $\Delta_\zeta$ in \eqref{xx-k0} according to their join and
meet. Fix a pair of partitions $\sigma\subseteq\rho$, and consider the
fiber
\[
\mathcal F(\rho,\sigma)
=
\{(\nu,\gamma):\nu\vee\gamma=\rho,\ \nu\wedge\gamma=\sigma\}.
\]
Note that for two pairs of partitions $\sigma^1\subseteq\rho^1$ and $\sigma^2\subseteq\rho^2$, if $(\sigma^1,\rho^1)\neq (\sigma^2,\rho^2)$, then $$\mathcal F(\rho^1,\sigma^1)\cap\mathcal F(\rho^2,\sigma^2)=\emptyset.$$
Hence \eqref{xx-k0} can be split into disjoint  summands  according to $\mathcal F(\rho,\sigma)$. 

Let $l'=\max\{\ell(\rho), \ell(\sigma)\}$, and set
\begin{equation}\label{an-2-zbnd}
I=\{1,2,\ldots, l'\}.
\end{equation}
For $E\subseteq I$, define two vectors $\nu_E$ and $\gamma_E=\nu_{I\setminus E}$ in $\mathbb Z_{\geq 0}^{l'}$ by 
\[
\nu_E(i)=
\begin{cases}
\rho_i,& i\in E,\\
\sigma_i,& i\notin E,
\end{cases}
\qquad
\gamma_E(i)=
\begin{cases}
\sigma_i,& i\in E,\\
\rho_i,& i\notin E.
\end{cases}
\]
Observe that  every pair $(\nu,\gamma)$ satisfying $\nu\vee\gamma=\rho$ and $\nu\wedge\gamma=\sigma$ 
is obtained in exactly  $2^{\#\{i:\rho_i=\sigma_i\}}$ ways by taking $E$ satisfying that 
\[
\{i\in I:\nu_i=\rho_i, \rho_i\ne \sigma_i\}\subseteq E\subseteq\{i\in I:\nu_i=\rho_i\}.
\] Hence the contribution of $\mathcal F(\rho,\sigma)$ is
\begin{align}
\Delta_{\zeta;\rho,\sigma}
&=\sum_{(\nu,\gamma)\in \mathcal F(\rho,\sigma)}\bigl(R_\zeta(\nu,\gamma)-L_\zeta(\nu,\gamma)\bigr)
s_\nu s_\gamma
=\sum_{E\subseteq I}
O_\zeta(E)\,
s_{\nu_E}s_{\gamma_E}\nonumber\\[5pt]
&=\sum_{E\subseteq I}
O_\zeta(E)\,
s_{\nu_E}s_{\nu_{I\setminus E}},\label{k-0-123}
\end{align}
where
\[
O_\zeta(E)
=
2^{-\#\{i:\rho_i=\sigma_i\}}(R_\zeta(\nu_E,\gamma_E)-L_\zeta(\nu_E,\gamma_E)).
\]
Here, if   $\nu_E$ or $\gamma_E$ is not a partition,  the
corresponding   terms are understood as zero. 
Now, to confirm the Schur positivity of $\Delta_\zeta$, it is enough to show that each $\Delta_{\zeta;\rho,\sigma}$ is
Schur positive.

Step 3.   We apply    Theorem  \ref{cor:straight-complementary-orchestra} to prove the Schur positivity of $\Delta_{\zeta;\rho,\sigma}$ in \eqref{k-0-123}. 
It suffices to  verify  that $O_\zeta$ satisfies the  orchestra
inequalities required in Theorem  \ref{cor:straight-complementary-orchestra}.
Fix any  set partition $B_1\sqcup\cdots\sqcup B_t$ of the set $I$ in  \eqref{an-2-zbnd}.
For $H\subseteq[t]$, let
\[
E_H=\bigcup_{h\in H}B_h,
\qquad
\nu_H=\nu_{E_H},
\qquad
\gamma_H=\gamma_{E_H}=\nu_{E_{\overline{H}}}.
\]
The  orchestra
inequality we  need to prove is 
\[
\sum_{H\subseteq[t]}
O_\zeta(E_H)\ge0,
\]
which is equivalent to (both sides multiplied   by the   factor $2^{\#\{i:\rho_i=\sigma_i\}}$)
\[
\sum_{H\subseteq[t]}
[\mathbf q^\zeta]\,
K_{\lambda\vee\mu,\nu_H}(\mathbf q)
K_{\lambda\wedge\mu,\gamma_H}(\mathbf q)
\ge
\sum_{H\subseteq[t]}
[\mathbf q^\zeta]\,
K_{\lambda,\nu_H}(\mathbf q)
K_{\mu,\gamma_H}(\mathbf q).
\]
We will prove a stronger version
\begin{equation}\label{ghk-nk}
\sum_{H\subseteq[t]}
K_{\lambda,\nu_H}(\mathbf q)
K_{\mu,\gamma_H}(\mathbf q)
\le^{\mathbf q}
\sum_{H\subseteq[t]}
K_{\lambda\vee\mu,\nu_H}(\mathbf q)
K_{\lambda\wedge\mu,\gamma_H}(\mathbf q).
\end{equation}

We use   Theorem  \ref{WRLS} to give a proof of \eqref{ghk-nk}. 
By Step 1, we see that  there exists an integer $m$ such that every partition
$\nu$ or $\gamma$ appearing in a nonzero term of $\Delta_\zeta$
satisfies
$\ell(\nu),\ell(\gamma)\le m$. Set $k=\max\{\l_1,\mu_1\}$. 
Consider the finite distributive lattice 
 $\OS_k^{l,m}$, as defined in \eqref{aghkhn-afn-0}. 
For each $H\subseteq[t]$, define subsets of $\OS_k^{l,m}$ by
\begin{alignat*}{2}
  A_H &= \OS^l(\lambda\to\nu_H), & \qquad B_H &= \OS^l(\mu\to\nu_H), \\[5pt]
  C_H &= \OS^l(\lambda\vee\mu\to\nu_H), & \qquad D_H &= \OS^l(\lambda\wedge\mu\to\nu_H).
\end{alignat*}
Here when $\nu_H$ is not a valid partition, the above sets are understood as empty sets. 
Then
\[
\sum_{S\in A_H}\mathbf q^{\mathbf r(S)}
=
K_{\lambda,\nu_H}(\mathbf q),
\qquad
\sum_{T\in B_H}\mathbf q^{\mathbf r(T)}
=
K_{\mu,\nu_H}(\mathbf q),
\]
and similarly for $C_H,D_H$.

It remains to  verify the hypotheses  in    Theorem  \ref{WRLS}.
Take $S\in A_H$ and $S'\in B_J.$
The  initial shapes of $S\vee S'$ and $S\wedge S'$ are $\lambda\vee\mu$ and $\lambda\wedge\mu$,  and their final shapes are $\nu_H\vee\nu_J$ and $\nu_H\wedge\nu_J$. 
 By the construction of $\nu_H$, one has
\[
\nu_H\vee\nu_J=\nu_{H\cup J},
\qquad
\nu_H\wedge\nu_J=\nu_{H\cap J}.
\]
Therefore
\[
A_H\vee B_J\subseteq C_{H\cup J},
\qquad
A_H\wedge B_J\subseteq D_{H\cap J}.
\]
Moreover, we see  from \eqref{abkb-8-i}  that   $\mathbf r=(r_1,\ldots,r_{2l-1})$ are modular functions. 
Applying  Theorem  \ref{WRLS} to these four families, we have 
\[
\sum_{H\subseteq[t]}
\left(\sum_{S\in A_H}\mathbf q^{\mathbf r(S)}\right)
\left(\sum_{T\in B_{\overline{H}}}\mathbf q^{\mathbf r(T)}\right)
\le^{\mathbf q}
\sum_{H\subseteq[t]}
\left(\sum_{U\in C_H}\mathbf q^{\mathbf r(U)}\right)
\left(\sum_{V\in D_{\overline{H}}}\mathbf q^{\mathbf r(V)}\right),
\]
which, along with $\gamma_H=\nu_{\overline{H}}$,   is exactly \eqref{ghk-nk}.

Collecting the above  analysis in  Steps 1, 2 and 3, we conclude that for every fixed $\zeta$, $\Delta_\zeta$ in \eqref{xx-k0} is Schur positive. 
This enables us to complete  the proof.
\end{myproof}

\footnotesize{

\textsc{\{Peter L. Guo, Mingyang Kang, Jiaji Liu\} Center for Combinatorics, Nankai University, LPMC, Tianjin 300071, P.R. China}

{\it
Email address: \tt lguo@nankai.edu.cn, 2120240005@mail.nankai.edu.cn,\\ 1120240006@mail.nankai.edu.cn}

\end{document}